\documentclass[journal]{IEEEtran}

\usepackage{amsmath,amssymb,amsfonts}
\usepackage{algorithmic}
\usepackage{graphicx}
\usepackage{algorithm,algorithmic}
\usepackage{hyperref}
\usepackage{amsthm}
\usepackage{xcolor}

\usepackage{booktabs} 
\usepackage{multirow}
\usepackage{threeparttable}
\newtheorem{theorem}{Theorem}
\newtheorem{lemma}{Lemma}

\newtheorem{assumption}{Assumption}
\newtheorem{remark}{Remark}

\begin{document}

\title{Guaranteeing Both Consensus and Optimality in Decentralized Nonconvex Optimization with Multiple Local Updates}


\author{Jie Liu, Zuang Wang, and Yongqiang Wang, ~\IEEEmembership{Senior Member,~IEEE} \thanks{The work was supported in part by the National Science Foundation under Grants CCF-2106293, CCF-2215088, CNS-2219487, CCF-2334449, and CNS-2422312 (Corresponding author: Yongqiang Wang).} \thanks{The authors are with the Department of Electrical and Computer Engineering, Clemson University, Clemson, SC 29634, USA (e-mail: jie9@clemson.edu; zuangw@clemson.edu; yongqiw@clemson.edu).}    }


\maketitle

\begin{abstract}
Scalable decentralized optimization in modern large-scale systems hinges on efficient communication. A widely adopted approach to reducing communication overhead in distributed optimization is performing multiple local updates between two consecutive communication rounds, as is commonly implemented in server-assisted distributed learning frameworks such as federated learning. However, extending this approach to fully decentralized settings presents fundamental challenges. In fact, all existing decentralized algorithms that incorporate multiple local updates can guarantee accurate convergence only under strong convexity assumptions, limiting their applicability to nonconvex optimization problems commonly found in machine learning. Moreover, many methods require exchanging and storing auxiliary variables—like gradient tracking vectors or correction terms—to ensure convergence under data heterogeneity, incurring significant communication and memory costs. In this paper, we propose \textsc{MILE}, a fully decentralized algorithm that guarantees both optimality and consensus under multiple local updates in general nonconvex settings. This is enabled by our novel periodic-system-based formulation and a lifting-based analysis, which together allow us to derive a closed-form expression for the state evolution over multiple local updates—a theoretical result not achieved by existing methods. This closed-form expression allows us, for the first time, to establish guaranteed consensus and optimality in decentralized nonconvex optimization under multiple local updates—in sharp contrast to existing results, which only ensure optimality of the average state in the nonconvex setting. We prove that \textsc{MILE} attains an \( O(1/T) \) convergence rate under both exact and stochastic gradient settings.  In addition, MILE requires each interacting pair of agents to exchange only a single variable, which minimizes communication and memory overhead compared with counterpart algorithms. Numerical experiments on benchmark datasets confirm the effectiveness of the proposed algorithm.

\end{abstract}


\IEEEpeerreviewmaketitle

\section{Introduction}
Over the past decade, decentralized optimization has been extensively studied and successfully applied in a variety of domains, including power systems \cite{power_system1,power_system2,power_system3}, neural network training~\cite{neural_network1_distributed_optimization,neural_network2_distributed_optimization,pmlr-v54-mcmahan17a}, smart grids~\cite{smart_grid1_distributed_optimization,smart_grid2_distributed_optimization,smart_grid3_distributed_optimization}, and wireless networks~\cite{wireless_network1_distributed_optimization,wireless_network2_distributed_optimization,wireless_network3_distributed_optimization}. Unlike centralized approaches that rely on collecting all data at a central server for processing, decentralized optimization methods enable a network of agents to collaboratively solve a global optimization problem using only local computations and communication with neighboring agents. This architecture offers several notable benefits, including enhanced robustness to failures, improved data privacy, and superior scalability \cite{yao_yang_survey}. These advantages have driven a growing interest in decentralized optimization across various fields in recent years \cite{Leishi3,youcheng_niu1,xiuxianli1,zhanyudistributed,demingyuan1,yuan2024distributed}.

In decentralized optimization, efficient information exchange among agents is critical—particularly in high-dimensional applications such as deep learning, where optimization variables can scale to billions of parameters \cite{highdimension1,highdimension2,highdimension3}. A major limitation of many existing decentralized algorithms is their reliance on message exchanges between neighboring agents at every iteration. In high-dimensional scenarios, such as those encountered in deep learning, this communication demand can become a significant bottleneck. To address this challenge, federated learning commonly employs a different strategy: clients perform multiple local updates on their local data before periodically communicating with a central server to synchronize model parameters \cite{pmlr-v119-karimireddy20a,pmlr-v119-malinovskiy20a,pmlr-v130-charles21a,NEURIPS2020_4ebd440d,A_Mitra_Linear,pmlr-v202-huang23p,NEURIPS2022_28553688}. While this strategy has proven effective in federated learning, extending it to fully decentralized optimization presents substantial challenges: decentralized optimization applications lack a central server to correct the drift in local optimization variables caused by independent local updates of individual agents. This drift leads local variables to converge toward the optima of individual local objective functions, rather than a global optimum, thereby impeding overall convergence. Consequently, the limited existing methods that attempt to address this issue face several fundamental limitations.

\begin{itemize}
    \item \textbf{Lack of Consensus Guarantee:} Recent studies \cite{LED,taolin1,LOCAL1_AVERAGING} have explored decentralized optimization schemes that incorporate multiple local updates. However, these methods typically only guarantee the convergence of the \emph{average} of the local variables to a global optimum. This is inadequate in fully decentralized settings—where agents lack access to the global average and must rely solely on their own local variables.  Without ensuring that each agent’s variable converges to a common global optimum, the final outcomes may be inconsistent across the network, compromising both the reliability and practical deployability of the algorithm. It is worth noting that recent works have established consensus guarantee for \emph{strong convexity} objective functions \cite{10886043,LOCAL_SMALL_STEPSIZE,WeiShi3,Nedich_time_varying2,Mind_submitted}. However, extending such consensus guarantees to more complex nonconvex problems remains a fundamental open challenge.

    \item \textbf{High Communication and Memory Overhead:} Most existing decentralized algorithms that can ensure optimality under multiple local updates \cite{taolin1,LOCAL1_AVERAGING,LOCAL_SMALL_STEPSIZE,WeiShi3,Nedich_time_varying2,zhou2025decentralized,kunyuan1} require agents to exchange and store auxiliary variables—such as gradient-tracking variables or correction terms—that are typically of the same dimension as the optimization variable. This leads to substantial communication overhead and imposes significant memory demands on each agent. Such costs pose serious challenges in high-dimensional applications, particularly in decentralized machine learning applications where the dimension of optimization variables can scale to billions.

\end{itemize}

In this paper, we propose an algorithm that guarantees both consensus and optimality by that all agents’ states converge to a common stationary point, while ensuring communication and memory efficiency by eliminating the need for auxiliary variables. The main \textbf{contributions} are summarized as follows:

\begin{itemize}
\item \textbf{New Algorithm:} We propose {MILE}, a novel decentralized optimization algorithm that supports multiple local updates and guarantees both consensus and optimality in nonconvex settings—\emph{without relying on auxiliary variables}. Inspired by the well-known decentralized algorithm EXTRA~\cite{WeiShi1}, {MILE} employs an update rule that incorporates gradient information from both the current and previous iterations, enabling each agent to independently correct the drift introduced by multiple local updates. It is worth noting that unlike EXTRA \cite{WeiShi1} and its variants ED~\cite{ED_Yuan}, NIDS~\cite{Ming_Yan1}, and $D^2$~\cite{Ming_Yan2} that only consider one local update between two consecutive communications and often diverge when incorporated with multiple local updates, {MILE} is specifically designed to compensate for the drifts caused by multiple local updates. A key novelty is a carefully chosen weight parameter, which plays an essential role in guaranteeing both consensus and optimality in decentralized nonconvex optimization with multiple local updates. Numerical experiments on benchmark datasets confirm the effectiveness of the proposed algorithm.

\item \textbf{Consensus and Optimality Guarantee:} Unlike existing nonconvex methods that only guarantee convergence of the average of local variables to a stationary point under multiple local updates \cite{LED,taolin1,LOCAL1_AVERAGING}, MILE ensures that each agent’s local variable converges to a common stationary point with an \( O(1/T) \) convergence rate in general nonconvex settings. To the best of our knowledge, MILE is the first algorithm to achieve both optimality and consensus in decentralized nonconvex optimization under multiple local updates. To support this analysis, we develop new analytical tools based on lifting techniques, which enable us to rigorously prove convergence of all agents' states to a common stationary point in decentralized nonconvex optimization with multiple local updates. Specifically, we propose a novel analytic framework by recasting the decentralized optimization algorithms with multiple local updates into an equivalent periodic dynamical system. We then employ the lifting technique to derive a \emph{closed-form analytical expression} for the optimization variable (see Lemma~\ref{lemma_recursive}, a core contribution of this paper). The combination of periodic system modeling and lifting-based analysis offers a novel framework that, to the best of our knowledge, has not been reported before, and enables, for the \emph{first time}, the establishment of both consensus and optimality in decentralized nonconvex optimization with multiple local updates. Moreover, our theoretical guarantees extend to the stochastic gradient setting where each agent's gradients are subject to noises.



\item \textbf{Efficient Communication:} In contrast to existing methods \cite{taolin1,LOCAL1_AVERAGING,LOCAL_SMALL_STEPSIZE,kunyuan1}, which require each pair of interacting agents to exchange two 
$n$-dimensional vectors per communication round—typically the local optimization variable and an auxiliary gradient-tracking or correction term—our MILE reduces this cost by requiring each agent to transmit only a single 
$n$-dimensional vector. This reduction in communication overhead per round makes MILE particularly appealing for bandwidth-constrained or large-scale decentralized learning scenarios.


\item \textbf{Efficient Memory:} In addition to its communication efficiency, MILE is also memory-efficient. Each agent only needs to store two 
$n$-dimensional vectors—its current and previous optimization variables. In contrast, counterpart algorithms such as K-GT \cite{taolin1} require each agent to maintain three or more auxiliary vectors. This reduction in memory overhead makes MILE particularly suitable for deployment on resource-constrained devices. Table \ref{sample-table2} provides a detailed comparison of the memory and communication requirements between MILE and existing counterpart algorithms.



\end{itemize}

\section{Preliminaries}\label{problem_section}
\subsection{Notations}
We use \(\mathbb{Z}\), \(\mathbb{Z}^{+}\), \(\mathbb{R}^n\), and \(\mathbb{R}^{m\times n}\) to denote integers, positive integers, real \(n\)-dimensional vectors, and real \((m\times n)\)-dimensional matrices, respectively. We use \(\mathbf{1}_N\in\mathbb{R}^N\) and \(\mathbf{0}_N\in\mathbb{R}^N\) to denote the \(n\)-dimensional vectors with all elements being \(1\) and $0$, respectively. We use \(\mathbf{I}_{N}\in\mathbb{R}^{N\times N}\) to denote the identical matrix.  \([x]_i\) and \([A]_{ij}\) denote  the \(i^{th}\) element of  \(x\in\mathbb{R}^{n}\) and the \((i,j)^{th}\) element of  \(A\in\mathbb{R}^{n\times m}\), respectively. We represent the Euclidean norm of \(x\in\mathbb{R}^n\) as \(\Vert x\Vert=\sqrt{\sum^n_{j=1}[x]^2_j}\) and the Frobenius norm of $A\in\mathbb{R}^{n\times m}$ as $\Vert A\Vert_{F}=\sqrt{\sum^{n}_{i=1}\sum^{m}_{j=1}[A]^2_{ij}}$. We denote the transposes of \(y\in\mathbb{R}^n\) and \(A\in\mathbb{R}^{m\times n}\) as \(y^{\bf T}\) and \(A^{\bf T}\), respectively. For two symmetric matrices \(A,B\in\mathbb{R}^{N\times N}\), we use \(A\succ B\) (\(A\succeq B\)) to denote that \(A-B\) is positive definite (positive semidefinite). Given two positive integers \(a\) and \(b\), we use \(a\ \textbf{mod}\ b\) to represent the remainder of the division of \(a\) by \(b\). We use \({O}(c(t))\)  to represent sequences \(d(t)\) satisfying \(\limsup_{t\rightarrow+\infty}|\frac{d(t)}{c(t)}|<\infty\).

\subsection{Problem Settings}

We consider an undirected network \( \mathcal{G} = \{\mathcal{V}, \mathcal{E}\} \), where \( \mathcal{V} = \{1, 2, \dots, N\} \) denotes the set of agents and \( \mathcal{E} \subseteq \mathcal{V} \times \mathcal{V} \) represents the set of undirected edges. An undirected edge \( (i, j) \in \mathcal{E} \) implies a bidirectional communication link between agents \( i \) and \( j \), allowing them to exchange information. All agents cooperatively solve the following global optimization problem:
\begin{align}\label{op_problem}
\min_{x \in \mathbb{R}^n} f(x) = \frac{1}{N} \sum_{i=1}^{N} f_i(x),
\end{align}
where \( f_i : \mathbb{R}^n \rightarrow \mathbb{R} \) denotes the local objective function only accessible to agent \( i \in \mathcal{V} \). We make the following standard assumption on the local objective functions $f_i(x)$:
\begin{assumption}\label{smooth_assumption}
The objective function $f_i(x)$ is $L$-smooth over $\mathbb{R}^n$, that is, there exists a constant $L>0$ such that $$\Vert \nabla f_i(x)-\nabla f_i(y)\Vert\leq L \Vert x-y\Vert$$ holds for any $x,y\in\mathbb{R}^n$.
\end{assumption}
From Assumption \ref{smooth_assumption} and the definition (\ref{op_problem}) of the global objective function $f(x)$, we can easily obtain that $f(x)$ also satisfies the $L$-smooth property over $\mathbb{R}^n$. In addition, we make the following standard assumption to make sure that (\ref{op_problem}) has at least one solution:
\begin{assumption}\label{non_empty}
The optimal solution set 
\begin{align*}
\mathcal{X}^*=\{x^*\in\mathbb{R}^n |x^*=\arg\min_{x\in\mathbb{R}^n} f(x)\}
\end{align*}
is not empty, i.e., there exists at least one $x^*\in\mathbb{R}^n$ such that $f(x^*)\leq f(x)$  holds for any $x\in\mathbb{R}^n$. 
\end{assumption}
Assumptions \ref{smooth_assumption} and \ref{non_empty} are mild and widely adopted in the decentralized nonconvex optimization literature \cite{LED,taolin1,Ming_Yan2}. They are more general than imposing convexity or the Polyak-Lojasiewicz condition.  In addition, we make the following standard assumption \cite{WeiShi1,Ming_Yan1,WeiShi2} on the communication network \( \mathcal{G} = \{\mathcal{V}, \mathcal{E}\} \).
\begin{assumption}\label{mixing_matrix} The undirected network $\mathcal{G}=\{\mathcal{V},\mathcal{E}\}$ is connected. The mixing matrix  $W\in\mathbb{R}^{N\times N}$ satisfies:
\begin{itemize}
    \item [1)] If $i\neq j$ and $(i,j)\notin\mathcal{E}$, then $w_{ij}=0$;
    \item [2)] $W=W^{\bf T}$;
    \item [3)]  ${\bf Null}(\mathbf{I}_N-W)={\bf span}(\mathbf{1}_N)$;
    \item [4)] $2\mathbf{I}_N \succeq W+\mathbf{I}_N \succ\mathbf{0}_{N\times N}$.
\end{itemize}
\end{assumption}
From Assumption \ref{mixing_matrix}, we know that the mixing matrix $W$ has $N$ real eigenvalues satisfying 
\begin{align}\label{eigenvalues_property}
1=\lambda_1>\lambda_2\geq \dots\geq \lambda_N>-1.
\end{align}
Next, we present our new algorithm and rigorously characterize its convergence performance.

\section{Algorithm Description}\label{step_details}

In this section, we propose a decentralized nonconvex optimization algorithm  (MILE) that has \underline{m}ultiple \underline{i}terations of \underline{l}ocal updat\underline{e}s between two consecutive communications. In the next section, we will prove that MILE can ensure both consensus and optimality in the nonconvex case, which has not been achieved  before under multiple local updates. MILE is inspired by decentralized optimization algorithms EXTRA~\cite{WeiShi1}, ED~\cite{ED_Yuan}, NIDS~\cite{Ming_Yan1}, and $D^2$~\cite{Ming_Yan2}, but with a fundamental difference. Specifically, these existing algorithms \cite{WeiShi1,ED_Yuan,Ming_Yan1,Ming_Yan2} only perform one local update between two consecutive communications. Directly incorporating multiple local updates in these methods often leads to divergence, as noted in Remark~2 of~\cite{LED}. In contrast, by judiciously designing a weight parameter, MILE ensures accurate convergence even when multiple local updates are incorporated.





Some notations should be explained before introducing our algorithm. The stepsize and the number of local updates of  {MILE} are denoted as $\alpha>0$ and $\tau\in\mathbb{Z}^{+}$, respectively. The optimization variable of agent $i\in\mathcal{V}$ at iteration time $t$ is denoted as $x_i(t)$. Due to the recursive local update mechanism, two initial values, \(x_i(-1)\) and \(x_i(0)\), are required. For any $i\in\mathcal{V}$, the value  \(x_i(-1)\) can be arbitrarily chosen in \(\mathbb{R}^n\), whereas \(x_i(0)\) should be set according to the following rule:
\begin{align}\label{initial_value_design1}
x_i(0)=x_i(-1)-\alpha \nabla f_i(x_i(-1)).
\end{align}
Now, we are in a position to present our algorithm {MILE}.
\begin{algorithm}[htpb]
   \caption{ {MILE} }
   \label{algorithm_recur}
\begin{algorithmic}
\STATE{\textbf{Initialization}: number of local updates $\tau\in\mathbb{Z}^{+}$, stepsize $\alpha$, and weight parameter $\xi\in(0,\frac{2}{\tau+3})$. Agent $i$ sets $\widetilde{w}_{ii}=(1-\xi)+\xi w_{ii}$ and $\widetilde{w}_{ij}=\xi w_{ij}$ for any $i\in\mathcal{V}$ and $j\neq i$.}
   \FOR{$t=0$ {\bfseries to} $T$}
   \FOR{each agent $i=1,2,\cdots,N$ in parallel}
   
   \IF{$t\ \textbf{mod}\ \tau =0$}\STATE 
   Each agent $i$ receives information from its neighbors and updates its local variable as
   \begin{align}\label{equation_state_new_condition}
       x_{i}(t&+1)=\sum^{N}_{j=1}\widetilde{w}_{ij} \Big\{2x_j(t)-x_j(t-1)\nonumber\\
       &-\alpha \nabla f_j(x_j(t))+\alpha \nabla f_j(x_j(t-1))\Big\}.
   \end{align}
   \ELSE{ \STATE Each agent $i$ does local update}
   \begin{align}\label{equation_state_new_condition2}
       x_{i}(t+1)= &2x_i(t)-x_i(t-1)-\alpha \nabla f_i(x_i(t))\nonumber\\
       &+\alpha \nabla f_i(x_i(t-1)).
   \end{align}
   \ENDIF
   \ENDFOR
   \ENDFOR
\end{algorithmic}
\end{algorithm}

In {MILE}, at each communication round (i.e., when \( t\ \textbf{mod}\ \tau = 0 \)), agent~\(i\) transmits a single \(n\)-dimensional vector:
\begin{align}\label{linear_combination}
2x_i(t) - x_i(t-1) - \alpha \nabla f_i(x_i(t)) + \alpha \nabla f_i(x_i(t-1))
\end{align}
to its neighboring agents. Upon receiving this vector, each agent performs \(\tau\) local updates according to \eqref{equation_state_new_condition} and~\eqref{equation_state_new_condition2} in Algorithm~\ref{algorithm_recur}. A key component of {MILE}'s update rule is the weight parameter \( \xi \in (0, \frac{2}{\tau + 3} ) \), which influences the update through the $\widetilde{w}_{ij}$ parameter in \eqref{equation_state_new_condition}. Selecting \( \xi \) within the prescribed interval is essential to ensure the exact convergence of Algorithm~\ref{algorithm_recur}. The theoretical basis for this selection is detailed in Section \ref{importance_weight_parameter}, drawing on matrix decomposition theory.

{MILE} is communication-efficient. At each communication round, each agent transmits only a single \( n \)-dimensional vector (see~\eqref{linear_combination}) to its neighboring agents. In contrast, existing methods with multiple local updates in \cite{taolin1,WeiShi3,kunyuan1} require transmitting two \( n \)-dimensional vectors between two interacting agents in every communication round: the local optimization variable and an auxiliary variable.

{MILE} is also memory-efficient, despite requiring storage of the optimization variable from the previous iteration. For example, in K-GT~\cite{taolin1}, each agent maintains three $n$-dimensional vectors: the local optimization variable, a gradient-tracking auxiliary variable, and another auxiliary variable for drift correction (note that we do not consider the storage of gradients, since they can be computed directly from the optimization variable). In contrast, {MILE} requires each agent to store only two $n$-dimensional vectors: the current and previous local optimization variables. Table~\ref{sample-table2} summarizes the memory and communication requirements of {MILE} and existing counterpart algorithms.

\section{Convergence Analysis}\label{full_batch}
In this section, we provide a comprehensive convergence analysis of {MILE}. Specifically, section~\ref{matrix_form} derives the equivalent matrix representation of Algorithm~\ref{algorithm_recur}. Section~\ref{consensus_analysis} establishes the consensus of all agents. Section~\ref{Convergence_Property_of_DoRecu_full} establishes the optimality of {MILE}. Finally, Section~\ref{importance_weight_parameter} provides a theoretical justification for selecting the weight parameter \( \xi \in (0, \frac{2}{\tau + 3} ) \) in Algorithm~\ref{algorithm_recur}.

\subsection{Matrix Form of {MILE} }\label{matrix_form}
We define the following matrices and vectors:
\begin{equation}\label{matrix_definition}
\left\{
\begin{aligned}
 X(t)&= [x_1(t), x_2(t), \cdots, x_N(t)] \in \mathbb{R}^{n \times N}, \\
\overline{\nabla f}(X(t)) &= \frac{1}{N} \sum_{i=1}^N \nabla f_i(x_i(t)),\  \overline{X}(t)= \frac{1}{N} \sum_{i=1}^N x_i(t),\\
\nabla f(X(t)) &= [\nabla f_1(x_1(t)), \cdots, \nabla f_N(x_N(t))].
\end{aligned}
\right.
\end{equation}
We also define a time-varying mixing matrix \( W(t) \) as follows:
\begin{equation}\label{matrix_W}
W(t) = 
\begin{cases}
(1 - \xi)\mathbf{I}_N + \xi W, & \text{if } t = k\tau, \\
\quad\quad\quad \mathbf{I}_N, & \text{otherwise},
\end{cases}
\end{equation}
where \( \xi \in (0,\frac{2}{\tau+3} )\) is the weight parameter and \( \tau \in \mathbb{Z}^{+} \) denotes the number of local updates.

Using the definitions in \eqref{matrix_definition} and \eqref{matrix_W}, the update rules \eqref{equation_state_new_condition} and \eqref{equation_state_new_condition2} in Algorithm~\ref{algorithm_recur} can be equivalently rewritten as
\begin{align}
X(t{+}1) &= 2X(t)W(t) - X(t{-}1)W(t) \nonumber\\
&\quad - \alpha \nabla f(X(t))W(t) + \alpha \nabla f(X(t{-}1))W(t). \label{convergence_analysis_process1}
\end{align}
Multiplying both sides of \eqref{convergence_analysis_process1} by \( \frac{1}{N} \mathbf{1}_N \) and using Assumption~\ref{mixing_matrix} (i.e., \( W(t)\mathbf{1}_N = \mathbf{1}_N \)), we obtain
\begin{align}
\overline{X}(t{+}1) &= 2\overline{X}(t) - \overline{X}(t{-}1) \nonumber\\
&\quad - \alpha \overline{\nabla f}(X(t)) + \alpha \overline{\nabla f}(X(t{-}1)). \label{averaging_property_X}
\end{align}
Rearranging terms in \eqref{averaging_property_X} yields
\begin{align}
\overline{X}(t{+}1) - \overline{X}(t) 
&= \overline{X}(t) - \overline{X}(t{-}1) \nonumber\\
&\quad - \alpha\big( \overline{\nabla f}(X(t)) - \overline{\nabla f}(X(t{-}1)) \big). \label{result_without_stochastic_1}
\end{align}
Applying mathematical induction to \eqref{result_without_stochastic_1} gives
\begin{align}
\overline{X}(t{+}1) - \overline{X}(t) 
&= \overline{X}(0) - \overline{X}(-1) \nonumber\\
&\quad - \alpha\big( \overline{\nabla f}(X(t)) - \overline{\nabla f}(X(-1)) \big). \label{result_without_stochastic}
\end{align}
Combining \eqref{initial_value_design1} and \eqref{result_without_stochastic} yields
\begin{align}
\overline{X}(t{+}1) = \overline{X}(t) - \alpha \, \overline{\nabla f}(X(t)). \label{average_xt_form}
\end{align}
According to Assumption~\ref{smooth_assumption}, \eqref{average_xt_form} further implies
\begin{align}
f(\overline{X}(t{+}1)) 
&\leq f(\overline{X}(t)) - \alpha \langle \nabla f(\overline{X}(t)), \overline{\nabla f}(X(t)) \rangle \nonumber\\
&\quad + \frac{L \alpha^2}{2} \Vert \overline{\nabla f}(X(t)) \Vert^2. \label{convergence_analysis_process2_1}
\end{align}
Applying the equality \( \|a + b\|^2 = \|a\|^2 + \|b\|^2 + 2\langle a, b \rangle \) to the right-hand side of \eqref{convergence_analysis_process2_1} leads to 
\begin{align}
&f(\overline{X}(t{+}1)) \leq f(\overline{X}(t)) - \bigl( \frac{\alpha}{2} - \frac{L \alpha^2}{2} \bigr) \Vert \overline{\nabla f}(X(t)) \Vert^2  \nonumber\\
& - \frac{\alpha}{2} \Vert \nabla f(\overline{X}(t)) \Vert^2 + \frac{\alpha L^2}{2N} \sum_{i=1}^N \Vert \overline{X}(t) - x_i(t) \Vert^2. \label{convergence_analysis_process2}
\end{align}
Summing both sides of~\eqref{convergence_analysis_process2} from \( t = 1 \) to \( T \) and rearranging terms, we can obtain
\begin{align}
&\sum_{t=1}^T \left\{ \Vert \nabla f(\overline{X}(t)) \Vert^2 + (1 - L\alpha) \Vert \overline{\nabla f}(X(t)) \Vert^2 \right\} \nonumber\\
\leq& \frac{2}{\alpha} \left( f(\overline{X}(1)) - f(x^*) \right) + \frac{L^2}{N} \sum_{t=1}^T \sum_{i=1}^N \Vert \overline{X}(t) - x_i(t) \Vert^2. \label{important_requiring_consensus}
\end{align}

\subsection{Consensus Analysis}\label{consensus_analysis}
We define $\widetilde{W} = (1 - \xi)\mathbf{I}_N + \xi W$ whose  eigenvalues $\{\rho_1, \rho_2, \cdots, \rho_N\}$ can be verified to satisfy 
\begin{align}\label{definition_rhoi_from_lambdai}
\rho_i = 1 - \xi + \xi \lambda_i,
\end{align}
where $\{\lambda_i\}_{i=1}^N$ are the eigenvalues of $W$. From Assumption~\ref{mixing_matrix} and \eqref{matrix_W}, $W(t)$ can be equivalently represented using an orthogonal matrix 
\begin{align}\label{orthoghnal_matrix}
P = [v_1, v_2, \cdots, v_N] \in \mathbb{R}^{N \times N}
\end{align}
satisfying $P^{\top}P = \mathbf{I}_N$ and a diagonal matrix $\Lambda(t)$ as follows:
\begin{align}\label{matrix_form_time_varying}
W(t) = P \Lambda(t) P^{\top},
\end{align}
where 
\begin{align}\label{definition_Lambda}
\Lambda(t) = \mathrm{diag}\{\rho_1(t), \rho_2(t), \cdots, \rho_N(t)\}
\end{align}
and
\begin{equation}\label{definition_rho_add}
\rho(t)=\left\{
\begin{aligned}
\rho_i, \quad &t = k \tau,\\
1,\quad  & t \neq k \tau.
\end{aligned}
\right.
\end{equation}
From \eqref{definition_rhoi_from_lambdai}, \eqref{orthoghnal_matrix}, \eqref{matrix_form_time_varying}, \eqref{definition_Lambda}, \eqref{definition_rho_add}, and  Assumption \ref{mixing_matrix}, we can obtain
\begin{align}\label{lambda1_rho1_property}
\lambda_1=\rho_1=1\quad {\text{and}}\quad v_1 = \frac{1}{\sqrt{N}} \mathbf{1}_N.
\end{align}
Moreover, if we define
\begin{equation}\label{definition_y_h}
\left\{
\begin{aligned}
Y(t) &= X(t)P = [y_1(t), y_2(t), \cdots, y_N(t)],\\
H(t) &= {\nabla f}(X(t))P = [h_1(t), h_2(t), \cdots, h_N(t)],
\end{aligned}
\right.
\end{equation}
from \eqref{convergence_analysis_process1} and \eqref{matrix_form_time_varying}, we arrive at
\begin{align}\label{convergence_analysis_process3}
Y(t+1) = &\, 2Y(t)\Lambda(t) - Y(t-1)\Lambda(t) - \alpha H(t)\Lambda(t) \nonumber\\
& + \alpha H(t-1)\Lambda(t).
\end{align}
Using \eqref{definition_Lambda}, \eqref{definition_rho_add}, and \eqref{definition_y_h}, \eqref{convergence_analysis_process3} can be equivalently rewritten as
\begin{align}\label{convergence_analysis_process4}
y_i(t+1) =& \rho_i(t) \bigl( 2y_i(t) - y_i(t-1) - \alpha h_i(t) \nonumber\\
&+ \alpha h_i(t-1) \bigr)
\end{align}
for any $i = 1, 2, \cdots, N$.

We denote the set of $n$-dimensional unit vectors as $\{e_i\}_{i=1}^n$, i.e., 
\begin{equation}\label{definition_e}
[e_i]_j=\left\{
\begin{aligned}
1, \quad & j = i,\\
0,\quad  & j \neq i.
\end{aligned}
\right.
\end{equation}
Next, we analyze the consensus error by expressing it in terms of the variables \( y_i(t) \). From the definitions \eqref{matrix_definition} and \eqref{definition_e}, we have
\begin{align}\label{add_euqation1_explantion}
\sum_{i=1}^N \Vert \overline{X}(t) - x_i(t) \Vert^2 = \sum_{i=1}^N \Vert X(t) e_i - \frac{1}{N} X(t) \mathbf{1}_N \Vert^2.
\end{align}
Based on the definition of $\Vert\cdot\Vert_{F}$, \eqref{orthoghnal_matrix}, \eqref{lambda1_rho1_property}, and the orthogonal property $P P^\top=\mathbf{I}_N$ , \eqref{add_euqation1_explantion} further implies
\begin{align}\label{add_euqation3_explantion}
\sum_{i=1}^N \Vert \overline{X}(t) - x_i(t) \Vert^2 = \left\Vert X(t) P \left(\mathbf{I}_N - e_1 e_1^\top \right) P^\top \right\Vert_F^2.
\end{align}
Substituting the definition of $Y(t)$ given in \eqref{definition_y_h} into  \eqref{add_euqation3_explantion} yields
\begin{align*}
\sum_{i=1}^N \Vert \overline{X}(t) - x_i(t) \Vert^2 = \left\Vert Y(t) \left(\mathbf{I}_N - e_1 e_1^\top \right) \right\Vert_F^2,
\end{align*}
which further implies
\begin{align}\label{consensus_analysis_important}
\sum_{i=1}^N \Vert \overline{X}(t) - x_i(t) \Vert^2 = \sum_{i=2}^N \Vert y_i(t) \Vert^2.
\end{align}
As a result, obtaining a tractable expression for \( y_i(t) \) for the recursive relation \eqref{convergence_analysis_process4} becomes essential for analyzing the consensus behavior of {MILE}. To facilitate the analysis, we introduce Lemma \ref{lemma_recursive}, which is crucial for establishing, for the first time, both consensus and optimality in decentralized nonconvex optimization with multiple local updates.




\begin{lemma}\label{lemma_recursive}
For a scalar sequence $\{a(t)\}^{\infty}_{t=0}$ satisfying the recursive relation
\begin{align}\label{lemma_recursive1}
a(t+1)={\rho}(t)\bigl(2a(t)-a(t-1)+b(t)-b(t-1)\bigr),
\end{align}
where 
\begin{equation}\label{definition_rho}
{\rho}(t)=\left\{
\begin{aligned}
\ \rho, &\quad t=k\tau,\\
\  1, &\quad t\neq k\tau,
\end{aligned}
\right.
\end{equation}
$\frac{\tau-1}{\tau+3}<\rho<1$, and $k\in\mathbb{Z}$, we have the following results for $1\leq p\leq \tau$ (with $a(0)\in\mathbb{R}$ and $a(1)\in\mathbb{R}$ being initial values of the sequence)
\begin{align}
a(k\tau+p)=&pa(k\tau+1)-(p-1)a(k\tau)\nonumber\\
&+\sum^{p-1}_{q=1}\sum^{q}_{j=1}\bigl\{b(k\tau+j)-b(k\tau+j-1)\bigr\},\label{analytical_formula1}\\
a(k\tau+1)=&(\sqrt{\rho})^k\bigl(F_{11}(k)a(1)+F_{12}(k)a(0)\bigr)\nonumber\\
&+\sum^{k-1}_{s=0}(\sqrt{\rho})^{k-1-s}\bigl\{F_{11}(k-1-s)G_{11}(s)\nonumber\\
&+F_{12}(k-1-s)G_{21}(s)\bigr\},\label{analytical_formula2}\\
a(k\tau)=&(\sqrt{\rho})^k\bigl(F_{21}(k)a(1)+F_{22}(k)a(0)\bigr)\nonumber\\
&+\sum^{k-1}_{s=0}(\sqrt{\rho})^{k-1-s}\bigl\{F_{21}(k-1-s)G_{11}(s)\nonumber\\
&+F_{22}(k-1-s)G_{21}(s)\bigr\},\label{analytical_formula3}
\end{align}
where $\cos(\theta)=\frac{\rho(\tau+1)+1-\tau}{2\sqrt{\rho}}$,
$\sin(\theta)=\frac{\sqrt{4\rho-[\rho(\tau+1)+(1-\tau)]^2}}{2\sqrt{\rho}}$,  
\begin{align*}
F_{11}(s)&=\frac{\sin((s+1)\theta)}{\sin(\theta)}+\frac{(\tau-1)\sin(s\theta)}{\sqrt{\rho}\sin(\theta)},\\
F_{12}(s)&=-\frac{\tau\sqrt{\rho}\sin(s\theta)}{\sin(\theta)},\quad F_{21}(s)=\frac{\tau\sin(s\theta)}{\sqrt{\rho}\sin(\theta)},\\
F_{22}(s)&=- \frac{\sin((s-1)\theta)}{\sin(\theta)}-\frac{(\tau-1)\sin(s\theta)}{\sqrt{\rho}\sin(\theta)},\\
G_{11}(s)&={\rho}\sum^{\tau}_{j=1}j[b(s\tau+\tau+1-j)-b(s\tau+\tau-j)],\\
G_{21}(s)&=\sum^{\tau}_{j=1}(j-1)[b(s\tau+\tau+1-j)-b(s\tau+\tau-j)].
\end{align*}

\end{lemma}
\begin{proof}
The recursive relation in (\ref{lemma_recursive1}) can be equivalently expressed as the following matrix form:
\begin{align}\label{appendix1_recursive2}
x(t+1)=A(t)x(t)+B(t)u(t),
\end{align}
where
\begin{equation}\label{definition_xaub}
\left\{
\begin{aligned}
x(t)=&\left[                 
  \begin{array}{c}   
    a(t) \\ 
    a(t-1) \\ 
  \end{array}
\right], \quad    
A(t)=\left[                 
  \begin{array}{cc}   
    2{\rho}(t) & -{\rho}(t) \\ 
    1 & 0 \\ 
  \end{array}
\right],\\
u(t)=&\left[                 
  \begin{array}{c}   
    b(t) \\ 
    b(t-1) \\ 
  \end{array}
\right],\quad    
B(t)=\left[                 
  \begin{array}{cc}   
    {\rho}(t) & -{\rho}(t) \\ 
    0 & 0 \\ 
  \end{array}
\right].    
\end{aligned}
\right.
\end{equation}
Because $\rho(t)$ changes periodically, equation \eqref{appendix1_recursive2} defines a periodic system. By applying the lifting approach from \cite{lifting}, (\ref{appendix1_recursive2}) can be equivalently expressed as the following system:
\begin{align}\label{appendix1_recursive3}
x((k+1)\tau+1)=F_1x(k\tau+1)+G_1u_1(k),
\end{align}
where 
\begin{equation}\label{definition_after_lifting}
\left\{
\begin{aligned}
G_1=&\left[                 
    C(1),  \cdots, C(j),  \cdots, C(\tau) 
\right]\in\mathbb{R}^{2\times 2\tau},\\
u_1(k)=&\left[        
   u^{\bf T}(k\tau+1), \cdots,u^{\bf T}(k\tau+\tau) 
\right]^{\bf T}\in\mathbb{R}^{2\tau},\\
F_1=&\prod^{\tau-1}_{i=0}A(\tau-i), \\
C(j)=&\left[                 
  \begin{array}{cc}   
    {\rho} (\tau-j+1) & -{\rho} (\tau-j+1) \\ 
    \tau-j & j-\tau \\ 
  \end{array}
\right].
\end{aligned}
\right.
\end{equation}

Applying mathematical induction to (\ref{appendix1_recursive3}) yields
\begin{align}\label{appendix1_recursive4}
x(k\tau+1)=F^{k}_1x(1)+\sum^{k-1}_{s=0}F^{k-1-s}_1 G_1 u_1(s).
\end{align}
Using the definitions of $G_1$ and $u_1(k)$ in \eqref{definition_after_lifting}, along with the definition of $u(t)$ in \eqref{definition_xaub}, we have
\begin{align}
  G_1 u_1(s)=\left[                 
  \begin{array}{c}   
   G_{11}(s) \\ 
   G_{21}(s)
  \end{array}
\right],\label{definition_G1}
\end{align} 
where
\begin{align*}
G_{11}(s)&={\rho}\sum^{\tau}_{j=1}j[b(s\tau+\tau+1-j)-b(s\tau+\tau-j)],\\
G_{21}(s)&=\sum^{\tau}_{j=1}(j-1)[b(s\tau+\tau+1-j)-b(s\tau+\tau-j)].
\end{align*}

Next, we establish the properties of  $F^{s}_1$ for $s=1,2,\cdots,k$. Substituting \eqref{definition_rho} and \eqref{definition_xaub} into \eqref{definition_after_lifting} yields
\begin{equation}\label{definition_F1}
F_1=\left[                 
  \begin{array}{cc}   
    {\rho}(\tau+1) & -{\rho}\tau \\ 
    \tau & -\tau+1 \\ 
  \end{array}
\right].          
\end{equation}
The eigenvalues $\mu_1$ and $\mu_2$ of the matrix $F_1$ satisfy 
\begin{align}\label{appendix1_quadratic_equation}
\mu^2-[\rho(\tau+1)+(1-\tau)]\mu +\rho=0.
\end{align}
Since the following relation holds for the quadratic equation in \eqref{appendix1_quadratic_equation} when $\rho$ satisfies $\frac{\tau-1}{\tau+3}<\rho<1$:
\begin{align*}
[\rho(\tau+1)+(1-\tau)]^2-4\rho<0,
\end{align*}
we can express the solutions to \eqref{appendix1_quadratic_equation} as 
\begin{align*}
\mu_1 &= \sqrt{\rho} \bigl( \cos(\theta) + i \sin(\theta) \bigr), \\
\mu_2 &= \sqrt{\rho} \bigl( \cos(\theta) - i \sin(\theta) \bigr),
\end{align*}
where
\begin{equation}\label{theta_definition}
\begin{cases}
\sin(\theta) = \dfrac{\sqrt{4\rho - [\rho(\tau + 1) + (1 - \tau)]^2}}{2 \sqrt{\rho}}, \\[10pt]
\cos(\theta) = \dfrac{\rho(\tau + 1) + 1 - \tau}{2 \sqrt{\rho}}.
\end{cases}
\end{equation}
From Euler's formula, the eigenvalues can be rewritten as 
\begin{align}\label{eigenvalue_lambda}
\mu_1 = \sqrt{\rho} e^{i\theta},\quad \mu_2 = \sqrt{\rho} e^{-i\theta},
\end{align}
with corresponding eigenvectors
\begin{equation}\label{v1v2_definition1}
v_1 = \begin{bmatrix}   
\dfrac{\sqrt{\rho} e^{i\theta} + \tau - 1}{\tau} \\ 
1  
\end{bmatrix}, \quad    
v_2 = \begin{bmatrix}   
\dfrac{\sqrt{\rho} e^{-i\theta} + \tau - 1}{\tau} \\ 
1  
\end{bmatrix}.
\end{equation}
Thus, we have
\begin{align}\label{v1v2_definition3}
F_1^s = [v_1, v_2]
\begin{bmatrix}   
\mu_1^s & 0 \\ 
0 & \mu_2^s 
\end{bmatrix}
[v_1, v_2]^{-1}
\end{align}
for any \( s = 0, 1, \ldots, k \), where
\begin{align}\label{v1v2_definition2}
[v_1, v_2]^{-1} = \frac{-i \tau}{2 \sqrt{\rho} \sin(\theta)}
\begin{bmatrix}   
1 & -\frac{\sqrt{\rho} e^{-i\theta} + \tau - 1}{\tau} \\ 
-1 & \frac{\sqrt{\rho} e^{i\theta} + \tau - 1}{\tau}
\end{bmatrix}.
\end{align}
From \eqref{v1v2_definition1}, \eqref{v1v2_definition2}, and \eqref{v1v2_definition3}, for any \( s \geq 0 \), we have
\begin{align}
F_1^s 
= & \frac{-i \tau (\sqrt{\rho})^s}{2 \sqrt{\rho} \sin(\theta)}
\begin{bmatrix}   
\frac{\sqrt{\rho} e^{i\theta} + \tau - 1}{\tau} & \frac{\sqrt{\rho} e^{-i\theta} + \tau - 1}{\tau} \\ 
1 & 1
\end{bmatrix} \nonumber \\
& \times
\begin{bmatrix}   
e^{i s \theta} & -\frac{\sqrt{\rho} e^{-i\theta} + \tau - 1}{\tau} e^{i s \theta} \\ 
- e^{-i s \theta} & \frac{\sqrt{\rho} e^{i\theta} + \tau - 1}{\tau} e^{-i s \theta}
\end{bmatrix}.
\label{F_property1}
\end{align}
For notational convenience, we express \( F_1^s \) as
\begin{align}
F_1^s 
= & (\sqrt{\rho})^{s}
\begin{bmatrix}   
F_{11}(s) & F_{12}(s) \\ 
F_{21}(s) & F_{22}(s) 
\end{bmatrix},
\label{F_property1_notation}
\end{align}
where the components \( F_{11}(s) \), \( F_{12}(s) \), \( F_{21}(s) \), and \( F_{22}(s) \) are specified and reexpressed below using Euler’s formula:
\begin{align}
F_{11}(s)=&\frac{-i\tau}{2\sqrt{\rho}\sin(\theta)}\big\{\frac{\sqrt{\rho}}{\tau}(e^{i(s+1)\theta}-e^{-i(s+1)\theta})\nonumber\\
&+\frac{\tau-1}{\tau}(e^{is\theta}-e^{-is\theta})\big\}\nonumber\\
=&\frac{\sin((s+1)\theta)}{\sin(\theta)}+\frac{(\tau-1)\sin(s\theta)}{\sqrt{\rho}\sin(\theta)},\label{F1_EULER}\\
F_{21}(s)=&\frac{-i\tau}{2\sqrt{\rho}\sin(\theta)}\{e^{s\theta i}-e^{-s\theta i}\}=\frac{\tau\sin(s\theta)}{\sqrt{\rho}\sin(\theta)},\label{F2_EULER}\\
F_{22}(s)=&\frac{-i\tau}{2\sqrt{\rho}\sin(\theta)}\big\{\frac{\sqrt{\rho}}{\tau}(e^{-i(s-1)\theta}-e^{i(s-1)\theta})\nonumber\\
&+\frac{\tau-1}{\tau}(e^{-is\theta}-e^{is\theta})\big\}\nonumber\\
=&- \frac{\sin((s-1)\theta)}{\sin(\theta)}-\frac{(\tau-1)\sin(s\theta)}{\sqrt{\rho}\sin(\theta)},\label{F3_EULER}\\
F_{12}(s)=&\frac{\sqrt{\rho}e^{-i\theta}+\tau-1}{\tau}\frac{\sqrt{\rho}e^{i\theta}+\tau-1}{\tau}\frac{-i\tau}{2\sqrt{\rho}\sin(\theta)}\bigl\{e^{-s\theta i}\nonumber\\
&-e^{s\theta i}\bigr\}\nonumber\\
=&\frac{-\sin(s\theta)}{\tau\sqrt{\rho}\sin(\theta)}\big\{(\tau-1+\sqrt{\rho}\cos(\theta))^2+{\rho}\sin^2(\theta)\big\}\nonumber\\
=&-\frac{\sqrt{\rho}\tau\sin(s\theta)}{\sin(\theta)},\label{F4_EULER}
\end{align}
where the third equality of \eqref{F4_EULER} follows from \eqref{theta_definition}.

From \eqref{appendix1_recursive4}, \eqref{definition_G1}, and \eqref{F_property1}, we have
\begin{align}
&\quad\left[                 
  \begin{array}{c}   
   a(k\tau+1) \\ 
   a(k\tau)
  \end{array}
\right]\nonumber\\
&=(\sqrt{\rho})^k\left[                 
  \begin{array}{cc}   
  F_{11}(k) & F_{12}(k) \\ 
   F_{21}(k)  & F_{22}(k) 
  \end{array}
  \right]\left[                 
  \begin{array}{c}   
   a(1) \\ 
   a(0)
  \end{array}
\right]+\sum^{k-1}_{s=0}(\sqrt{\rho})^{k-1-s}\nonumber\\
&\left[                 
  \begin{array}{c}   
  F_{11}(k-1-s)G_{11}(s)+F_{12}(k-1-s)G_{21}(s)\\ 
   F_{21}(k-1-s)G_{11}(s)+F_{22}(k-1-s)G_{21}(s)
  \end{array}
  \right].\label{final_matrix_form_lemma1}
\end{align}
Thus, \eqref{analytical_formula2} and \eqref{analytical_formula3} in Lemma \ref{lemma_recursive} can be obtained from  \eqref{F1_EULER}, \eqref{F2_EULER}, \eqref{F3_EULER}, \eqref{F4_EULER}, and \eqref{final_matrix_form_lemma1}.

Next, we proceed to prove  \eqref{analytical_formula1} in Lemma \ref{lemma_recursive}. Rearranging the terms of \eqref{lemma_recursive1} yields
\begin{align*}
&a(k\tau+q)-a(k\tau+q-1)=a(k\tau+q-1)\\
&\quad\quad-a(k\tau+q-2)+b(k\tau+q-1)-b(k\tau+q-2).
\end{align*}
Applying mathematical induction to the above equality yields
\begin{align*}
&a(k\tau+q)-a(k\tau+q-1)\\
=&a(k\tau+1)-a(k\tau)+\sum^{q-1}_{j=1}\bigl\{b(k\tau+j)-b(k\tau+j-1)\bigr\}
\end{align*}
for $1\leq q\leq \tau$. For any $1\leq p\leq \tau$, summing both sides of the above equation from \( h = 2 \) to \( p \) to obtain
\begin{align*}
a(k\tau+p)&-a(k\tau+1)=(p-1)\bigl\{a(k\tau+1)-a(k\tau)\bigr\}\\
&+\sum^{p-1}_{q=1}\sum^{q}_{j=1}\bigl\{b(k\tau+j)-b(k\tau+j-1)\bigr\}.
\end{align*}
Thus, \eqref{analytical_formula1} in Lemma \ref{lemma_recursive} is established, which completes the proof (note \eqref{analytical_formula2} and \eqref{analytical_formula3} have been established in \eqref{final_matrix_form_lemma1}).
\end{proof}

From Lemma \ref{lemma_recursive}, we can derive the analytical expression for \( y_i(t) \) from the recursive formula~\eqref{convergence_analysis_process4}. To be specific, for any $1\leq p\leq\tau$ and $2\leq i\leq N$, we have
\begin{equation}\label{analytical_formula_of_y}
\left\{
\begin{aligned}
y_i(k\tau&)=(\sqrt{\rho_i})^k\bigl(F_{21,i}(k)y_i(1)+F_{22,i}(k)y_i(0)\bigr)\\
&+\sum^{k-1}_{s=0}(\sqrt{\rho_i})^{k-1-s}\bigl\{F_{21,i}(k-1-s)G_{11,i}(s)\\
&+F_{22}(k-1-s)G_{21,i}(s)\bigr\},\\
y_i(k\tau&+1)=(\sqrt{\rho_i})^k\bigl(F_{11,i}(k)y_i(1)+F_{12,i}(k)y_i(0)\bigr)\\
&+\sum^{k-1}_{s=0}(\sqrt{\rho_i})^{k-1-s}\bigl\{F_{11,i}(k-1-s)G_{11,i}(s)\\
&+F_{12}(k-1-s)G_{21,i}(s)\bigr\},\\
y_i(k\tau&+p)=py_i(k\tau+1)-(p-1)y_i(k\tau)\\
&-\alpha\sum^{p-1}_{q=1}\sum^{q}_{j=1}\bigl\{h_i(k\tau+j)-h_i(k\tau+j-1)\bigr\},
\end{aligned}
\right.
\end{equation}
where  $ F_{21,i}(s)=\frac{\tau\sin(s\theta_i)}{\sqrt{\rho_i}\sin(\theta_i)}$, $F_{12,i}(s)=\frac{-\tau\sqrt{\rho_i}\sin(s\theta_i)}{\sin(\theta_i)}$, $F_{11,i}(s)=\frac{\sin((s+1)\theta_i)}{\sin(\theta_i)}+\frac{(\tau-1)\sin(s\theta_i)}{\sqrt{\rho_i}\sin(\theta_i)}, \ F_{22,i}(s)= \frac{\sin((1-s)\theta_i)}{\sin(\theta_i)}$ $-\frac{(\tau-1)\sin(s\theta_i)}{\sqrt{\rho}\sin(\theta_i)},$  $\sin(\theta_i)=\frac{\sqrt{4\rho_i-[\rho_i(\tau+1)+(1-\tau)]^2}}{2\sqrt{\rho_i}}$, $\cos(\theta_i)=\frac{\rho_i(\tau+1)+1-\tau}{2\sqrt{\rho_i}}$, $G_{11,i}(s)={\alpha\rho}\sum^{\tau}_{j=1}j[h_i(s\tau+\tau-j)-h_i(s\tau+\tau+1-j)]$, and $G_{21,i}(s)=\alpha\sum^{\tau}_{j=1}(j-1)[h_i(s\tau+\tau-j)-h_i(s\tau+\tau+1-j)]$.

Next, we use \eqref{analytical_formula_of_y} to derive a bound on $\Vert y_i(t)\Vert^2$, which can then be used to bound the consensus error as described in \eqref{consensus_analysis_important}. To this end, we first need to bound $\Vert y_i(k\tau+1)\Vert$ and $\Vert y_i(k\tau+1)-y_i(k\tau)\Vert$. 

Using the first and the second relations in \eqref{analytical_formula_of_y}, we can obtain
\begin{align}
&\Vert y_i(k\tau+1)\Vert\nonumber\\
\leq& {\alpha}(\rho\tau+\tau-1) A_1 \sum^{k-1}_{s=0}(\sqrt{{\rho}})^{k-1-s}\Big\{\sum^{\tau}_{j=1}\Vert h_i(\tau s+1+\tau-j)\nonumber\\
&-h_i(\tau s+\tau-j)\Vert\Big\}+A_1 A_2(\sqrt{{\rho}})^k,\label{norm_analyize_form1}
\end{align}
and
\begin{align}
&\Vert y_i(k\tau+1)-y_i(k\tau)\Vert\nonumber\\
\leq&{\alpha}(\rho \tau+\tau-1) A_3\sum^{k-1}_{s=0}(\sqrt{\rho})^{k-1-s}\Big\{\sum^{\tau}_{j=1}\Vert h_i(s\tau+\tau+1-j)\nonumber\\
&-h_i(s\tau+\tau-j)\Vert\Big\}+A_3A_2(\sqrt{\rho})^k,\label{norm_analyize_form2}
\end{align}
where 
\begin{equation}\label{many_definition_A}
\left\{
\begin{aligned}
\rho=&1-\xi+\xi\lambda_2,\\
A_1=&\max_{2\leq i\leq N}\Big\{\frac{2\tau}{\sqrt{4{\rho_i}-[{\rho_i}(\tau+1)+1-\tau]^2}}\Big\},\\
A_2=&\max_{2\leq i\leq N}\bigl\{\Vert y_i(1)\Vert+\Vert y_i(0)\Vert\bigr\},\\
A_3=&\max\Big\{\frac{2\sqrt{\rho}+2}{\sqrt{4\rho-[\rho(\tau+1)+(1-\tau)]^2}},\\
&\frac{2\sqrt{\rho}(1+|\tau\sqrt{\rho}-\frac{\tau-1}{\sqrt{\rho}}|)}{\sqrt{4\rho-[\rho(\tau+1)+(1-\tau)]^2}}\Big\}.
\end{aligned}
\right.
\end{equation}

Moreover, using the third equality in \eqref{analytical_formula_of_y}, we have
\begin{align*}
&\Vert y_i(k\tau+p)\Vert\nonumber\\
\leq &\Vert y_i(k\tau+1)\Vert+(p-1)\Vert y_i(k\tau+1)-y_i(k\tau)\Vert\nonumber\\
&+{\alpha} \sum^{p-1}_{q=1}\sum^{q}_{j=1}\Vert h_i(k\tau+j)-h_i(k\tau+j-1)\Vert.
\end{align*}
Substituting  \eqref{norm_analyize_form1} and \eqref{norm_analyize_form2} into the above inequality yields
\begin{align}
&\Vert y_i(k\tau+p)\Vert\nonumber\\
\leq & pA_4 A_2(\sqrt{{\rho}})^k+{\alpha} \sum^{p-1}_{q=1}\sum^{q}_{j=1}\Vert h_i(k\tau+j)-h_i(k\tau+j-1)\Vert\nonumber\\
&+A_4p{\alpha}(\rho \tau+\tau-1) \sum^{k}_{s=1}(\sqrt{\rho})^{k-s}\sum^{\tau}_{j=1}\Vert h_i(s\tau-j)\nonumber\\
&-h_i(s\tau+1-j)\Vert, \label{norm_analyize_form3}
\end{align}
for any $1\leq p\leq \tau$, where 
\begin{align}
A_4=\max\{A_1,A_3\}.
\end{align}

Next, we proceed to analyze the cumulative consensus error over $(K+1)\tau$ iterations, which include $K+1$ communication rounds. To this end, we first characterize
\begin{align}\label{T_1_computation}
\sum^{T}_{t=1}\Vert y_i(t)\Vert^2=\sum^{K}_{k=0}\sum^{\tau}_{p=1}\Vert y_i(k\tau+p)\Vert^2.
\end{align}
\begin{lemma}\label{computation_of_yi}
There exist $B_2=3\tau^4+\frac{\tau^2(\tau+1)(2\tau+1)(\rho\tau+\tau-1)^2}{2(1-\sqrt{{\rho}})^2} A^2_4$ and $B_1=\frac{\tau(\tau+1)(2\tau+1) }{2(1-{\rho})}A^2_4 A^2_2$ such that
\begin{align}\label{lemma2_equation1}
\sum^{T}_{t=1}\Vert y_i(t)\Vert^2\leq B_1+{\alpha}^2B_2\sum^{T-1}_{t=1}\Vert h_i(t)-h_i(t-1)\Vert^2.
\end{align}
In addition, we also have 
\begin{align*}
\sum^{N}_{i=2}\Vert h_i(t-1)- h_i(t)\Vert^2\leq L^2\sum^{N}_{i=1} \Vert y_i(t)- y_i(t-1)\Vert^2.
\end{align*}
\end{lemma}
\begin{proof}
See Appendix \ref{proof_computation_of_yi}.
\end{proof}
Summing both sides of \eqref{lemma2_equation1} from $i=2$ to $N$ yields
\begin{align}\label{very_important_inequality2}
&\sum^{N}_{i=2}\sum^{T}_{t=1}\Vert y_i(t)\Vert^2\nonumber\\
\leq& NB_1 +{\alpha}^2L^2B_2\sum^{T-1}_{t=1}\sum^{N}_{i=1} \Vert y_i(t)- y_i(t-1)\Vert^2.
\end{align}
Then, we analyze $\Vert y_1(t)- y_1(t-1) \Vert^2$. From \eqref{orthoghnal_matrix}, \eqref{lambda1_rho1_property}, and \eqref{definition_y_h}, we can obtain
\begin{align}\label{y_1_t_overline}
y_1(t)=X(t)v_1=\sqrt{N} \overline{X}(t).
\end{align}
Combining \eqref{y_1_t_overline} and \eqref{average_xt_form} yields
\begin{align}\label{very_important_inequality3}
\Vert y_1(t+1)- y_1(t) \Vert^2\leq N{\alpha}^2\Vert\overline{\nabla f}(X(t))\Vert^2.
\end{align}
Substituting \eqref{very_important_inequality3} into \eqref{very_important_inequality2} leads to
\begin{align}
&\sum^{N}_{i=2}\sum^{T}_{t=1}\Vert y_i(t)\Vert^2\leq {\alpha}^4 L^2 N B_2 \sum^{T-1}_{t=1}\Vert\overline{\nabla f}(X(t-1))\Vert^2\nonumber\\
&\quad+NB_1+{\alpha}^2L^2B_2\sum^{T-1}_{t=1}\sum^{N}_{i=2} \Vert y_i(t)- y_i(t-1)\Vert^2.\label{add_property_inequality_more1}
\end{align}
Further using the inequality $\Vert a+b\Vert^2\leq 2\Vert a\Vert^2+\Vert b\Vert^2$ yields
\begin{align*}
&\sum^{N}_{i=2}\sum^{T}_{t=1}\Vert y_i(t)\Vert^2\nonumber\\
\leq & C_1+{\alpha}^2L^2B_2\sum^{T-1}_{t=1}\Big\{4\sum^{N}_{i=2} \Vert y_i(t)\Vert^2+{\alpha}^2 N\Vert\overline{\nabla f}(X(t))\Vert^2\Big\},
\end{align*}
where 
\begin{align}
C_1={\alpha}^2 L^2 N B_2 \bigr({\alpha}^2\Vert\overline{\nabla f}(X(0))\Vert^2+2A^2_2\bigl)+NB_1.\label{definition_C1}
\end{align}
Therefore, according to \eqref{consensus_analysis_important}, we have
\begin{align}
&\bigl(1-4{\alpha}^2L^2B_2\bigr)\sum^{T}_{t=1}\sum^{N}_{i=1}\Vert \overline{X}(t)-x_i(t)\Vert^2\nonumber\\
\leq&  C_1+{\alpha}^4 L^2 N B_2 \sum^{T-1}_{t=1}\Vert\overline{\nabla f}(X(t))\Vert^2,\label{consensus_property}
\end{align}
which establishes the relationship between consensus error $\sum^{T}_{t=1}\sum^{N}_{i=1}\Vert \overline{X}(t)-x_i(t)\Vert^2$ and {the average gradient} $\sum^{T-1}_{t=1}\Vert\overline{\nabla f}(X(t))\Vert^2$.


\begin{remark}
Lemma \ref{lemma_recursive} represents a central contribution of this paper, as it provides an explicit analytical expression for  $y_i(t)$ based on the recursive formulation in \eqref{convergence_analysis_process4}. To the best of our knowledge, this explicit analytical expression-based convergence analysis approach is entirely new and is pivotal for rigorously establishing both the convergence and consensus guarantees of MILE. Notably, the proof of Lemma \ref{lemma_recursive} introduces a novel analytical framework that, to the best of our knowledge, has not been previously applied in the context of distributed optimization. It begins by reformulating the original recursion into an equivalent periodic system (see \eqref{appendix1_recursive2}), which is then transformed into an equivalent static representation using the lifting technique \cite{lifting} (see \eqref{appendix1_recursive4}). Matrix decomposition theory is subsequently employed to analyze the resulting system. This unique combination of periodic system reformulation, lifting, and matrix decomposition constitutes a new and powerful proof strategy, enabling a precise characterization of MILE’s consensus and convergence properties. 
\end{remark}

\subsection{Optimality Analysis}\label{Convergence_Property_of_DoRecu_full}
Under a stepsize $\alpha$ satisfying $1-4{\alpha}^2L^2B_2>0$, substituting \eqref{consensus_property} into \eqref{important_requiring_consensus} yields
\begin{align}
& \sum^{T}_{t=1}\bigl\{\Vert \nabla f(\overline{X}(t)) \Vert^2+(1-{L{\alpha}}-\frac{{\alpha}^4 L^4 B_2}{B_3})\Vert \overline{\nabla f}(X(t))\Vert^2 \bigr\}\nonumber\\
\leq &\frac{2}{{\alpha}} \bigl\{ f(\overline{X}(1))-f(x^*)\bigr\}+ \frac{C_1 L^2}{B_3 N},\label{convergence_property_last_step}
\end{align}
where $B_3 = 1 - 4\alpha^2 L^2 B_2.$ Thus, we can derive the convergence rate of $\nabla f(\overline{X}(t))$.
\begin{theorem}\label{theorem1}
Under stepsize $0<{\alpha}\leq\frac{1}{\sqrt{5B_2}L}$ (see $B_2$ in Lemma \ref{computation_of_yi}), we have
\begin{align*}
&\frac{1}{T}\sum^{T}_{t=1}\Vert \nabla f(\overline{X}(t)) \Vert^2\nonumber\\
\leq& \frac{2}{\tau{\alpha} K} \bigl\{ f(\overline{X}(1))-f(x^*)\bigr\}+ \frac{ \Vert\overline{\nabla f}(X(0))\Vert^2}{\tau^3 K}\\
&+ \frac{4\Vert X(0)\Vert^2_{F}+4\Vert X(1)\Vert^2_{F}}{\tau^3\alpha^2 K}+\frac{ 2\Vert X(0)\Vert^2_{F}+2\Vert X(1)\Vert^2_{F}}{\alpha^2\rho^2\tau^4 K},
\end{align*}
where $T=\tau+\tau K$ is the total number of iterations (which include $K+1$ communication rounds).
\end{theorem}
\begin{proof}
See Appendix \ref{proof_theorem1}.
\end{proof}
Theorem~\ref{theorem1} establishes the convergence of the state average to a stationary point of the global objective at a rate of \( O(1/T) \) in general nonconvex settings. In fully decentralized optimization, this result is insufficient because individual agents lack direct access to the state average. To address this, we now combine the consensus properties obtained in \eqref{consensus_property} with Theorem \ref{theorem1} to establish the convergence of all local variables to a common stationary point.


\begin{theorem}\label{corollary1}
Under stepsize $0<{\alpha}\leq\frac{1}{\sqrt{5B_2}L}$  (see $B_2$ in Lemma \ref{computation_of_yi}), we have
\begin{align*}
&\frac{1}{NT}\sum^{N}_{i=1}\sum^{T}_{t=1}\Vert \overline{X}(t)-x_i(t)\Vert^2\nonumber\\
\leq& \frac{4{\alpha}}{\tau K} \bigl\{ f(\overline{X}(1))-f(x^*)\bigr\}+ \frac{ 3\alpha^2\Vert\overline{\nabla f}(X(0))\Vert^2}{\tau^3 K}\\
&+ \frac{12\Vert X(0)\Vert^2_{F}+12\Vert X(1)\Vert^2_{F}}{\tau^3 K}+\frac{ 6\Vert X(0)\Vert^2_{F}+6\Vert X(1)\Vert^2_{F}}{\rho^2\tau^4 K},
\end{align*}
where $T=\tau+\tau K$ is the total number of iterations (which include $K+1$ communication rounds).
\end{theorem}
\begin{proof}
See Appendix \ref{proof_corollary1}
\end{proof}
Theorems \ref{theorem1} and \ref{corollary1} establish the accurate convergence of MILE under multiple local updates in the nonconvex setting. This stands in sharp contrast to prior results in \cite{LED,taolin1,LOCAL1_AVERAGING}, which only guarantee convergence of the state average to a common stationary point. This distinction makes our results more applicable to fully decentralized settings, where the state average is not directly accessible to individual agents.


\subsection{The Importance of Weight Parameter}\label{importance_weight_parameter}

In this subsection, we explain why the weight parameter \( \xi \in (0, \frac{2}{\tau+3}) \) in Algorithm~\ref{algorithm_recur} is critical for ensuring the convergence of {MILE}. For clarity, we restate \eqref{appendix1_recursive4} in the proof of Lemma~\ref{lemma_recursive} below:
\begin{align*}
x(k\tau+1) = F_1^k x(1) + \sum_{s=0}^{k-1} F_1^{k-1-s} G_1 u_1(s).
\end{align*}
Based on the definition \eqref{definition_F1} of $F_1$ and properties of eigenvalue decomposition,  \( F_1^s \) can be expressed as
\begin{align}\label{matrix_decomposition}
F_1^s &= \left[                 
  \begin{array}{cc}   
    \rho(\tau+1) & -\rho\tau \\ 
    \tau & 1 - \tau \\ 
  \end{array}
\right]^s \nonumber \\
&= [v_1, v_2] 
\begin{bmatrix}
\mu_1^s & 0 \\
0 & \mu_2^s
\end{bmatrix}
[v_1, v_2]^{-1},
\end{align}
where \( \{\mu_1, \mu_2\} \) are the eigenvalues of \( F_1 \) (see \eqref{eigenvalue_lambda}), and \( \{v_1, v_2\} \) are the corresponding eigenvectors (see \eqref{v1v2_definition1}).

To avoid divergence of \( x(k\tau+1) \), we require the eigenvalues of \( F_1 \) to satisfy
\begin{align}\label{eigenvalue_property_lambda}
|\mu_1| < 1 \quad  \text{and} \quad |\mu_2| < 1.
\end{align}
To derive a sufficient condition on \( \rho \) that ensures~\eqref{eigenvalue_property_lambda}, we note that \( \mu_1 \) and \( \mu_2 \) are the roots of the following quadratic equation
\begin{align*}
\mu^2 - [\rho(\tau+1) + (1 - \tau)]\mu + \rho = 0.
\end{align*}
In the case of complex (non-real) roots, Vieta's formula implies
\begin{align*}
\mu_1 \mu_2 = \rho \quad  \text{and} \quad |\mu_1| = |\mu_2| = \sqrt{\rho},
\end{align*}
which guarantees the condition in \eqref{eigenvalue_property_lambda} since $\rho<1$ holds under the condition of Lemma \ref{lemma_recursive}. It is easy to verify that a sufficient condition for the roots to have non-zero imaginary parts is
\begin{align*}
\left[\rho(\tau+1) + (1 - \tau)\right]^2 - 4\rho < 0,
\end{align*}
which holds under $\frac{\tau - 1}{\tau + 3} < \rho < 1.$ Therefore,  from~\eqref{definition_rhoi_from_lambdai}, it is sufficient to require
\begin{align*}
\frac{\tau - 1}{\tau + 3} < 1 - \xi + \xi \lambda_i < 1
\end{align*}
for \( 2 \leq i \leq N \) to ensure that the eigenvalues of \( F_1 \) satisfy~\eqref{eigenvalue_property_lambda}. From Assumption~\ref{mixing_matrix}, \( 1 > \lambda_2 \geq \dots \geq \lambda_N > -1 \) holds. 
As a result, the following condition 
\begin{align*}
0 < \xi < \frac{2}{\tau + 3}
\end{align*}
is sufficient to ensure that the roots \( \mu_1 \) and \( \mu_2 \) satisfy \( |\mu_1| < 1 \) and \( |\mu_2| < 1 \).

Assume, to the contrary, that the weight parameter is not required—specifically, by setting \( \xi = 1 \), the matrix form of {MILE} becomes
\begin{align}\label{without_weight_parameter}
X(t+1) =\; & \big\{ 2X(t) - X(t-1) - \alpha \nabla f(X(t)) \nonumber \\
& + \alpha \nabla f(X(t-1)) \big\} W(t),
\end{align}
where \( W(t) = W \) for \( t = k\tau \), and \( W(t) = \mathbf{I}_N \) otherwise. It has been shown that the update rule in~\eqref{without_weight_parameter} is not guaranteed to converge and often diverges (see Remark~2 in~\cite{LED}). This further underscores the significance of our novel design in contrast to existing algorithms.

\section{Convergence Analysis with Stochastic Gradients}\label{mini_batch}
In Section~\ref{full_batch}, we established the convergence of {MILE} under exact gradients. In this section, we extend the convergence analysis to a more practical scenario where only stochastic gradient estimates are available. To be specific, we consider the local objective function \( f_i(x) \) defined as
\begin{align}
f_i(x) = \mathbb{E}_{\xi_i \sim D_i}[f_i(x, \xi_i)],
\end{align}
where \( \xi_i \) denotes a stochastic data sample drawn from a local data distribution \( D_i \), which may differ across agents. Consequently, agent \( i \) can only access a stochastic estimate \( \nabla f_i(x, \xi_i) \) of the true gradient \( \nabla f_i(x) \) for any \( x \in \mathbb{R}^n \). We adopt the following standard assumption on the stochastic gradient, which is widely used in the machine-learning literature (see, e.g., \cite{pmlr-v119-karimireddy20a,LED,Ming_Yan2,Shi_Pu_Tracking}):
\begin{assumption}\label{stochastic_gradient}
The stochastic gradient $\nabla f_i(x, \xi_i)$ is an unbiased estimate of the accurate gradient $\nabla f_i(x)$, with its variance bounded by $\sigma^2$. Specifically, we have 
\begin{align}
&\mathbb{E}_{\xi_i \sim D_i}[\nabla f_i(x,\xi_i)]=\nabla f_i(x),\label{noise1_assumption} \\
&\mathbb{E}_{\xi_i \sim D_i}[\Vert \nabla f_i(x,\xi_i)-\nabla f_i(x)\Vert^2]\leq \sigma^2,\label{noise2_assumption}
\end{align}
for any $x \in \mathbb{R}^n$ and $i \in \mathcal{V}$.
\end{assumption}

In this stochastic setting, the exact gradients \( \nabla f_i(x_i(t)) \) and \( \nabla f_i(x_i(t-1)) \) in updates (\ref{initial_value_design1}),~(\ref{equation_state_new_condition}), and~(\ref{equation_state_new_condition2}) of Algorithm~\ref{algorithm_recur} should be replaced with their stochastic counterparts \( \nabla f_i(x_i(t), \xi_i(t)) \) and \( \nabla f_i(x_i(t-1), \xi_i(t-1)) \), respectively, where \( \xi_i(t) \sim D_i \) denotes a sample drawn from the local data distribution at each iteration. We now present the convergence properties of {MILE} in this case.
\begin{theorem}\label{theorem2}
Under stepsize $0<\alpha\leq\frac{1}{\sqrt{13B_2}L}$ (see $B_2$ in Lemma \ref{computation_of_yi}), we have
\begin{align*}
& \frac{1}{T}\sum^{T}_{t=1}\mathbb{E}[\Vert \nabla f(\overline{X}(t))\Vert^2]\nonumber\\
\leq & \frac{2 \mathbb{E}[f(\overline{X}(1))]-2f(x^*)}{{\alpha} \tau K}+\frac{ 2\mathbb{E}[\Vert X(1)\Vert^2_{F}]+2\mathbb{E}[\Vert X(0)\Vert_F^2]}{\alpha^2\rho^2\tau^4K}\\
&+\frac{12\mathbb{E}[\Vert X(1)\Vert_F^2]+12\mathbb{E}[\Vert X(0)\Vert_F^2]}{\tau^3\alpha^2 K} +\frac{6  \mathbb{E}[\Vert\overline{\nabla f}(X(0))\Vert^2]}{\tau^3 K}\nonumber\\
&+ {L{\alpha}}\sigma^2+{156B_2}{\alpha^2}L^2\sigma^2,
\end{align*}
where $T=\tau+\tau K$ is the total number of iterations (which include $K+1$ communication rounds).
\end{theorem}

\begin{proof}
See Appendix \ref{proof_theorem2}.
\end{proof}
Theorem~\ref{theorem2} shows that due to the presence of noise in the gradients,  \( \frac{1}{T} \sum_{t=1}^{T} \mathbb{E} \Vert \nabla f(\overline{X}(t)) \Vert^2 \) converges to within a bounded region around zero, whose size is determined by the stepsize \( \alpha \), the number of local updates \( \tau \), smoothness constant \( L \), and the variance \( \sigma^2 \) of the stochastic gradients. We now present the corresponding consensus analysis of {MILE} under stochastic gradients.
\begin{theorem}\label{corollary2}
Under stepsize $0<\alpha\leq\frac{1}{\sqrt{13B_2}L}$ (see $B_2$ in Lemma \ref{computation_of_yi}), we have
\begin{align*}
&\frac{1}{NT}\sum^{N}_{i=1}\sum^{T}_{t=1}\mathbb{E}[\Vert \overline{X}(t)-x_i(t)\Vert^2]\\
\leq & \frac{24{\alpha} \{\mathbb{E}[f(\overline{X}(1))]-f(x^*)\}}{ \tau K}+\frac{78 \alpha^2 \mathbb{E}[\Vert\overline{\nabla f}(X(0))\Vert^2]}{\tau^3 K}\\
&+\frac{ 26\mathbb{E}[\Vert X(1)\Vert^2_{F}]+26\mathbb{E}[\Vert X(0)\Vert_F^2]}{\rho^2\tau^4K}+{2028B_2}{\alpha^2}\sigma^2\\
&+\frac{156\mathbb{E}[\Vert X(1)\Vert_F^2]+156\mathbb{E}[\Vert X(0)\Vert_F^2]}{\tau^3 K} + {6{\alpha}^2}\sigma^2,
\end{align*}
where $T=\tau+\tau K$ is the total number of iterations (which include $K+1$ communication rounds).
\end{theorem}

\begin{proof}
See Appendix \ref{proof_corollary2}.
\end{proof}

\section{Comparisons with Existing Works}\label{comparison_previous}

In this section, we systematically compare MILE with existing counterpart algorithms with multiple local updates.

\begin{table*}[htpb]
\caption{Comparison between MILE and Existing Algorithms with Multiple Local Updates}
\label{sample-table2}
\vskip -0.6in
\begin{center}
\begin{threeparttable}
\begin{small}
\begin{sc}
\begin{tabular}{c|c|c|c|c}
\toprule
 \multirow{2}*{Algorithm}  & Consensus & Convergence Analysis & Communication   & Memory    \\
  & Guarantee?  & in Nonconvex Settings? & Overhead\tnote{1} &   Overhead\tnote{2}\\
\midrule
   {MILE} (\emph{Ours}) & \(\raisebox{0.6ex}{\scalebox{0.7}{\(\sqrt{}\)}}\) & \(\raisebox{0.6ex}{\scalebox{0.7}{\(\sqrt{}\)}}\) &  1 & 2  \\
 LED \cite{LED}  & $\times$  & \(\raisebox{0.6ex}{\scalebox{0.7}{\(\sqrt{}\)}}\)  & 1  & 3   \\
 K-GT \cite{taolin1}  & $\times$ & \(\raisebox{0.6ex}{\scalebox{0.7}{\(\sqrt{}\)}}\) & 2  & 3   \\
Dec-FedTrack \cite{LOCAL1_AVERAGING}   &  \(\times\) & $\times$ &   2  &  3 \\
GTA \cite{LOCAL_SMALL_STEPSIZE} & \(\raisebox{0.6ex}{\scalebox{0.7}{\(\sqrt{}\)}}\) &  $\times$ &2  & 3  \\
 DIGing \cite{WeiShi3} & \(\raisebox{0.6ex}{\scalebox{0.7}{\(\sqrt{}\)}}\) &  $\times$ & 2  & 3  \\
 PPADCO \cite{Nedich_time_varying2}& \(\raisebox{0.6ex}{\scalebox{0.7}{\(\sqrt{}\)}}\)  & $\times$  & 2  & 3  \\
MIND \cite{Mind_submitted} & \(\raisebox{0.6ex}{\scalebox{0.7}{\(\sqrt{}\)}}\) & $\times$ &  1 & 2  \\
\bottomrule
\end{tabular}
\end{sc} 
\begin{tablenotes}    
        \footnotesize               
        \item[{1}] {To quantify communication overhead, we measure the number  of $n$-dimensional variables shared between two interacting agents during $\tau$ local updates.}
        \item[{2}] {To quantify memory overhead, we measure the number of $n$-dimensional variables that must be stored, where $n$ is the dimension of the optimization variable. It is worth noting that we do not consider gradients since they can be computed directly using the optimization variable.}
      \end{tablenotes} 
\end{small}
\end{threeparttable}
\end{center}
\vskip -0.1in
\end{table*}

\subsection{Consensus Guarantee in Nonconvex Settings}
Several recent works~\cite{LED,taolin1,LOCAL1_AVERAGING} have studied decentralized optimization with multiple local updates. However, these methods only guarantee the optimality of the average value of all agents' states. This is not sufficient in fully decentralized settings, where agents cannot access the average and must rely on their own local states. Without consensus, local solutions may differ, leading to inconsistent or suboptimal performance across the network. Some recent methods do ensure consensus \cite{10886043,LOCAL_SMALL_STEPSIZE,WeiShi3,Nedich_time_varying2}, but they rely on the assumption of \emph{strong convexity}, which limits their practical applicability—particularly for the nonconvex problems often encountered in machine learning. In contrast, our proposed algorithm, {MILE}, ensures both optimality and consensus under multiple local steps in general \emph{nonconvex} settings.   Table~\ref{sample-table2} presents a detailed comparison between {MILE} and existing algorithms in terms of their consensus guarantees and whether they require the strong convexity assumption.

\begin{remark}

It is worth noting that decentralized optimization with multiple local updates can be viewed as decentralized optimization under intermittent communications. This setting bears similarities to existing work on time-varying networks (e.g., \cite{WeiShi3,Nedich_time_varying2}). However, all existing distributed optimization results for time-varying networks—including our prior work \cite{Mind_submitted}—rely on the assumption of strong convexity. In contrast, this paper addresses the more general and challenging nonconvex setting.

\end{remark}

\subsection{Efficient Communication}
A key advantage of {MILE} is its communication-efficient design. Existing algorithms, such as K-GT~\cite{taolin1}, GTA \cite{LOCAL_SMALL_STEPSIZE}, and DiGing~\cite{WeiShi3}, typically require each agent to transmit two \( n \)-dimensional vectors per communication round: the local optimization variable and an auxiliary variable (e.g., a gradient-tracking or error-correction term). This results in higher communication overhead, particularly when the dimension $n$ is large, as in deep learning applications. In contrast, {MILE} only requires sharing one \( n \)-dimensional vector between two interacting agents per communication round.  As a result, {MILE} effectively reduces per-round communication costs by half compared to existing methods. Table~\ref{sample-table2} presents a detailed comparison of {MILE} and existing algorithms in terms of communication overhead.

\subsection{Efficient Memory}
In addition to communication efficiency, {MILE} also offers significant memory savings. Algorithms such as K-GT~\cite{taolin1}, LED~\cite{LED}, and DIGing~\cite{WeiShi3} require each agent to store at least three \( n \)-dimensional variables—typically including the optimization variable, a gradient-tracking variable, and an auxiliary variable for drift correction. This storage requirement can be demanding, especially in high-dimensional or resource-constrained environments. In contrast, {MILE} eliminates the need for auxiliary variables by employing a recursive structure that integrates both current and historical gradient information into the local updates. As a result, each agent only needs to store two variable: the current and previous optimization variables. This significantly reduces memory requirements without compromising convergence accuracy. Table~\ref{sample-table2} presents a detailed comparison of {MILE} and existing algorithms in terms of memory overhead.

\begin{remark}
MILE employs a fixed number $\tau$ of local updates between two consecutive communication rounds, in contrast to Scaffnew~\cite{scaffnew} and MG-Skip~\cite{luyao_guo}, which utilize a random number of local updates. This key difference—\emph{deterministic} versus \emph{randomized} communication skipping—not only leads to fundamentally different algorithmic behaviors but also requires distinct analytical techniques. Furthermore, \cite{scaffnew, luyao_guo} focus exclusively on the \emph{strongly convex} setting and only establish convergence \emph{in expectation}. In comparison, MILE is designed for the more general and complicated \emph{nonconvex} optimization setting and is equipped with an \emph{exact convergence} guarantee. MILE's more general applicability and stronger theoretical guarantee make it better suited for practical decentralized optimization and learning scenarios.
\end{remark}

\section{Numerical Experiments}\label{numerical_experiments}

{We evaluate our proposed algorithm by training a CNN on $10$ agents using the benchmark datasets MINST \cite{lecun1998mnist} and CIFAR-10 \cite{krizhevsky2009learning}. The CNN used for the MNIST dataset consists of two convolutional layers. The first convolutional layer has $16$ filters with a kernel size of $5$ and padding of $2$, followed by a $2\times 2$ max-pooling layer. The second convolutional layer has $32$ filters with the same kernel size and padding, also followed by $2\times 2$ max-pooling. The resulting feature maps are flattened and passed to a fully connected layer that outputs $10$ logits corresponding to the digit classes. The CNN for CIFAR-10 consists of three convolutional layers with $32$, $64$, and $128$ filters, respectively, each followed by a max-pooling layer. After the final convolutional and pooling layers, the network includes a fully connected layer with $256$ units, a dropout layer with a rate of $0.25$ for regularization, and a final dense output layer with $10$ units that produces the class logits. In our experiments, we compare the proposed algorithm against existing decentralized optimization methods allowing multiple local updates, including  LED \cite{LED},  K-GT \cite{taolin1}, and  DIGing \cite{WeiShi3}. Following \cite{hsu2019measuring} and \cite{Kim_Adpative}, we generate heterogeneous local datasets for $10$ agents under a ring topology, with the heterogeneity parameter set to $1$, which represents a high level of heterogeneity. The number of local updates was set to $\tau=10$ for both cases. }

\begin{figure}[H]
    \centering
    \begin{minipage}[t]{0.69\linewidth}
        \centering
        \includegraphics[width=\textwidth]{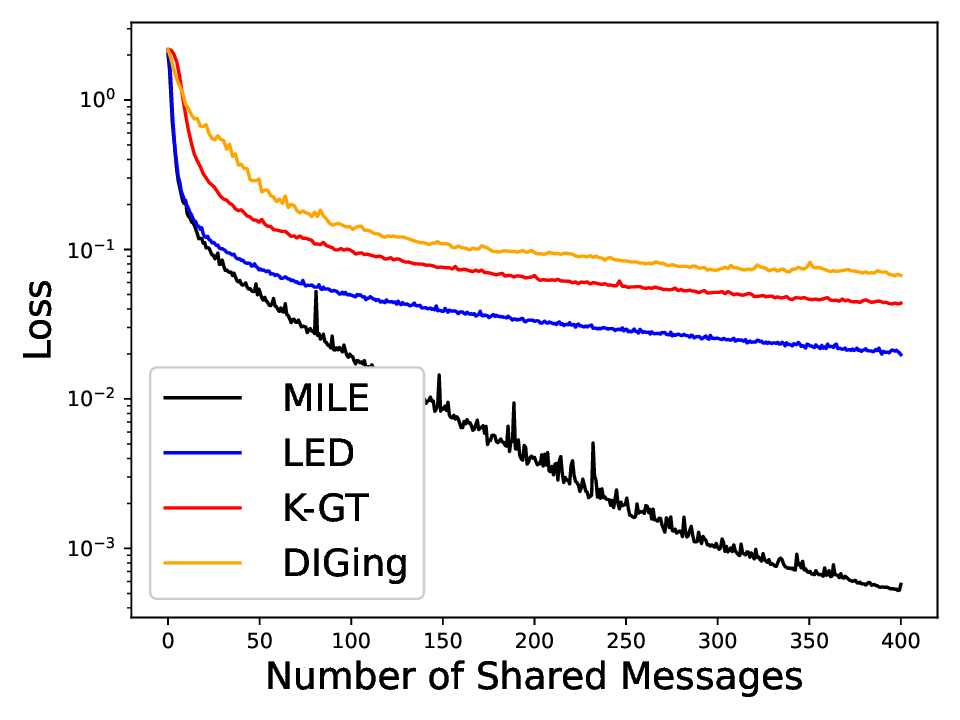}
        \centerline{\small{$\quad\ $ (a)}}
    \end{minipage}  
    \begin{minipage}[t]{0.71\linewidth}
        \centering
        \includegraphics[width=\textwidth]{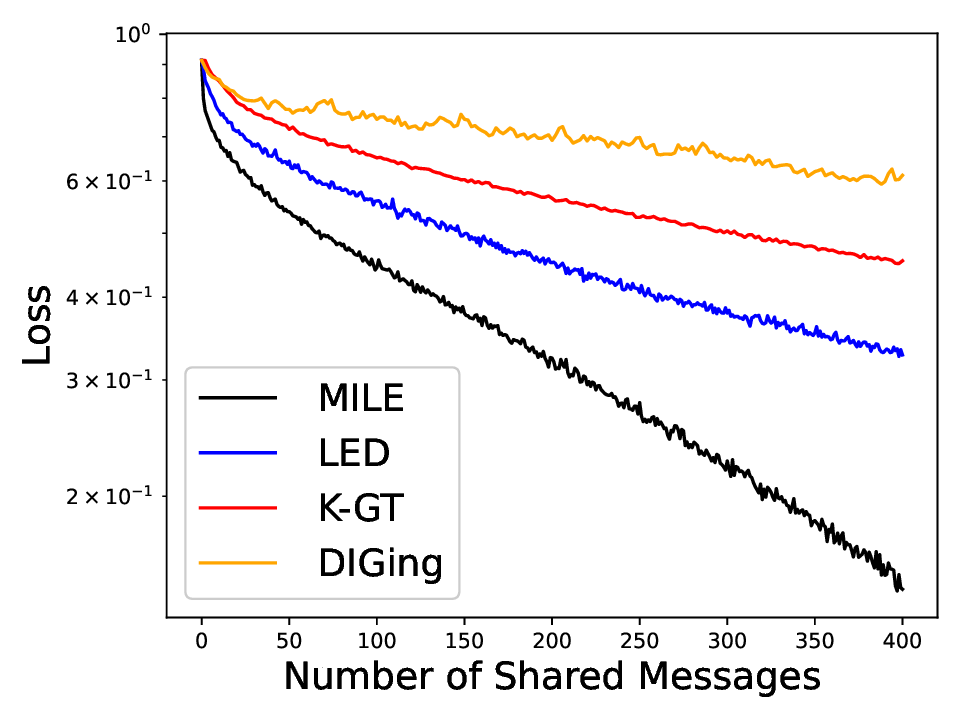}
        \centerline{\small{$\quad\ $ (b)}}
    \end{minipage}    
    \caption{Comparison of training loss under a common stepsize ($\alpha=0.12$ in subplot (a) and $\alpha=0.04$ in subplot (b)) between MILE and DIGing \cite{WeiShi3}, K-GT \cite{taolin1}, LED \cite{LED}. Subplot (a) shows the results on the MNIST dataset, while subplot (b) presents the results on the CIFAR-10 dataset. The number of local updates was set to $\tau=10$. Each curve represents the average of three independent runs.}
    \label{comparison}
\end{figure}

\begin{figure}[H]
        \centering
    \begin{minipage}[t]{0.7\linewidth}
        \centering
        \includegraphics[width=\textwidth]{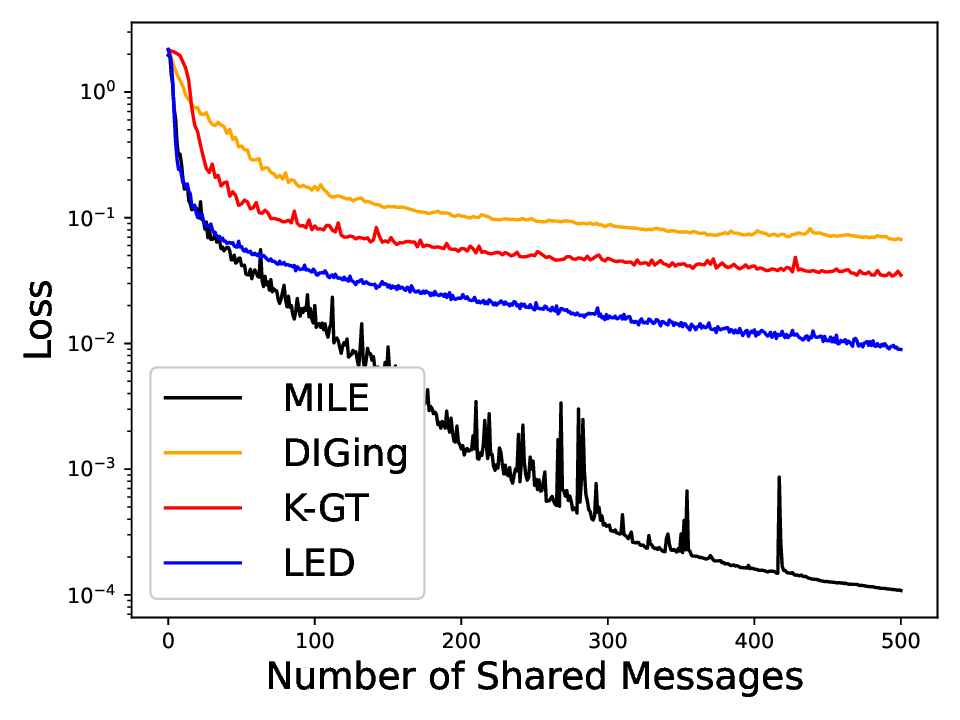}
        \centerline{$\quad\ $ \small{(a)}}
    \end{minipage}
    \begin{minipage}[t]{0.7\linewidth}
        \centering
        \includegraphics[width=\textwidth]{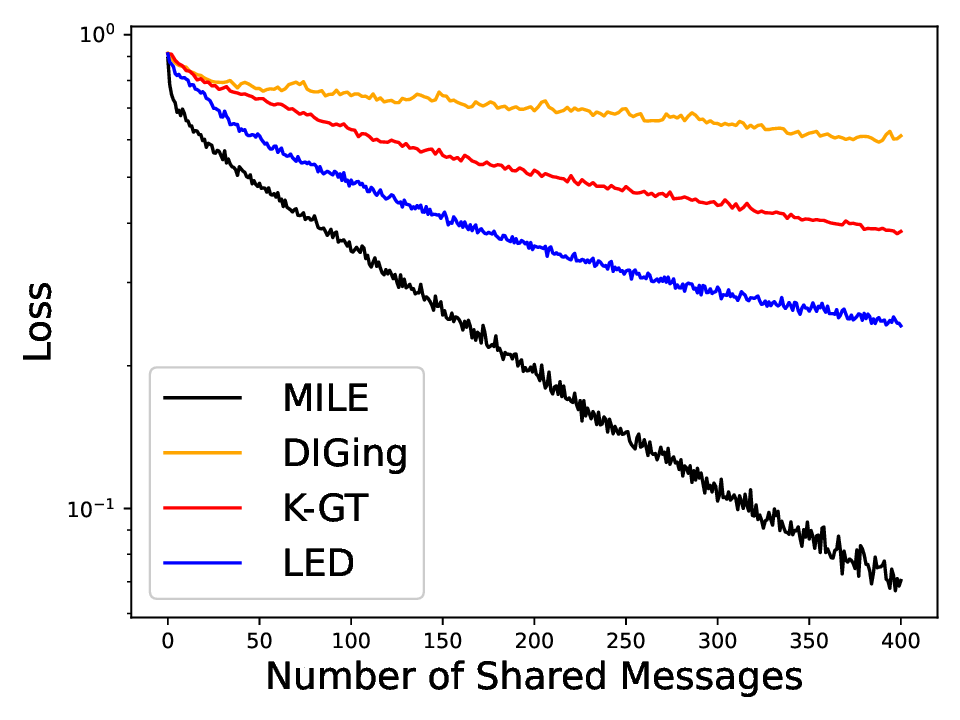}
        \centerline{$\quad\ $ \small{(b)}}
    \end{minipage}
    \caption{Comparison of training loss under the best-found stepsize for each algorithm. Subplot (a) shows the results on the MNIST dataset, while subplot (b) presents the results on the CIFAR-10 dataset. The number of local updates was set to $\tau=10$. Each curve represents the average of three independent runs.}
    \label{comparison2}
\end{figure}

{Figures~\ref{comparison} and~\ref{comparison2} present results under different stepsize settings. In Figures~\ref{comparison}~(subplot (a) for MNIST and subplot (b) for CIFAR-10), each curve represents the average of three independent runs using a common stepsize ($\alpha=0.12$ in subplot (a) and $\alpha=0.04$ in subplot (b)) across all algorithms. In Figure \ref{comparison2}~(subplot (a) for MNIST and subplot (b) for CIFAR-10), each curve shows the average of three independent runs using the best-found stepsize for each algorithm. As shown in Figures \ref{comparison} and~\ref{comparison2}, our algorithm achieves faster convergence than LED~\cite{LED}, K-GT~\cite{taolin1}, and DIGing~\cite{WeiShi3} on both the MNIST and CIFAR-10 datasets. These results confirm the effectiveness of the proposed algorithm.}

\section{Conclusion}\label{conclusion}
In this paper, we propose MILE, a fully decentralized optimization algorithm that supports multiple local updates and guarantees both optimality and consensus in general nonconvex settings—an achievement that, to the best of our knowledge, has not been reported before. In MILE, each agent communicates only one variable with its neighbors during each communication round and stores merely two local variables. This results in significant advantages over competing algorithms with multiple local updates in terms of communication and memory overhead. By formulating decentralized optimization with multiple local updates as a periodic system, we prove that MILE achieves an $O(1/T)$ convergence rate in both deterministic and stochastic gradient settings. A key enabler of this analysis is a set of novel theoretical tools we introduce, which leverage a lifting technique to derive explicit expressions for the recursive dynamics. To the best of our knowledge, this approach has not been previously explored and has broader implications beyond the algorithm proposed in this work. Machine learning experiments using benchmark datasets confirm the effectiveness of the proposed algorithm.

\appendices

\section{}\label{proof_computation_of_yi}
\subsection{Supporting Lemma for the Proof of Lemma \ref{computation_of_yi}}
\begin{lemma}[\cite{Ming_Yan2}]\label{ming_yan_lemma}
For two non-negative sequences $\{a(t)\}^{\infty}_{t=1}$ and $\{b(t)\}^{\infty}_{t=1}$ satisfying $a(t)=\sum^{t}_{s=1}\rho^{t-s}b(s)$ with $0\leq \rho<1$, we have
\begin{align*}
\sum^{k}_{t=s}a(t)\leq\sum^{k}_{t=s}\frac{b(s)}{1-\rho} \quad {\text and}\quad \sum^{k}_{t=s}a^2(t)\leq\sum^{k}_{t=s}\frac{b^2(s)}{(1-\rho)^2}.
\end{align*}
\end{lemma}

\subsection{Proof of Lemma \ref{computation_of_yi}}
Applying the inequality $\Vert a+b+c\Vert^2\leq 3\Vert a\Vert^2+3\Vert b\Vert^2+3\Vert c\Vert^2$ to \eqref{norm_analyize_form3} and \eqref{T_1_computation} yields
\begin{align}
&\sum^{T}_{t=1}\Vert y_i(t)\Vert^2\leq  3\sum^{K}_{k=0}\sum^{\tau}_{p=1}p^2A^2_4 A^2_2{\rho}^k\nonumber\\
&+3\sum^{K}_{k=0}\sum^{\tau}_{p=1}{\alpha}^2 \Big(\sum^{p-1}_{q=1}\sum^{q}_{j=1}\Vert h_i(k\tau+j)-h_i(k\tau+j-1)\Vert\Big)^2\nonumber\\
&+3A^2_4{\alpha}^2(\rho \tau+\tau-1)^2\sum^{\tau}_{p=1}p^2\sum^{K}_{k=0}\Big(\sum^{k}_{s=1}(\sqrt{\rho})^{k-s}\nonumber\\
& \quad\sum^{\tau}_{j=1}\Vert h_i(s\tau+1-j)-h_i(s\tau-j)\Vert\Big)^2.\label{norm_analyize_form4}
\end{align}

For the first term on the right-hand side of \eqref{norm_analyize_form4}, we have
\begin{align}
3\sum^{K}_{k=0}\sum^{\tau}_{p=1}p^2 A^2_4 A^2_2{\rho}^k\leq\frac{\tau(\tau+1)(2\tau+1) A^2_4 A^2_2}{2(1-{\rho})}\label{norm_analyize_form5}
\end{align}
based on the relation $\sum^{\tau}_{p=1}p^2=\frac{\tau(\tau+1)(2\tau+1)}{6}$.

For the second term on the right-hand side of  \eqref{norm_analyize_form4}, we have
\begin{align*}
&\sum^{K}_{k=0}\sum^{\tau}_{p=1}{\alpha}^2 \Big( \sum^{p-1}_{q=1}\sum^{q}_{j=1}\Vert h_i(k\tau+j)-h_i(k\tau+j-1)\Vert\Big)^2\nonumber\\
\leq & {\alpha}^2\tau^3\sum^{K}_{k=0} \Big( \sum^{\tau-1}_{j=1}\Vert h_i(k\tau+j)-h_i(k\tau+j-1)\Vert\Big)^2
\end{align*}
because of $p\leq\tau$ and $h\leq \tau$. Applying the relation $(\sum^{\tau}_{j=1}a_j)^2\leq\tau\sum^{\tau}_{j=1} a^2_j$ to the above inequality yields
\begin{align}
&\sum^{K}_{k=0}\sum^{\tau}_{p=1}{\alpha}^2 \Big( \sum^{p-1}_{q=1}\sum^{q}_{j=1}\Vert h_i(k\tau+j)-h_i(k\tau+j-1)\Vert\Big)^2\nonumber\\
\leq & {\alpha}^2\tau^4\sum^{K}_{k=0} \sum^{\tau-1}_{j=1}\Vert h_i(k\tau+j)-h_i(k\tau+j-1)\Vert^2.\label{norm_analyize_form7}
\end{align}

For the third term on the right-hand side of \eqref{norm_analyize_form4},  we have
\begin{align}
&\sum^{K}_{k=1} \Big(\sum^{k}_{s=1}(\sqrt{{\rho}})^{k-s}\sum^{\tau}_{j=1}\Vert h_i(\tau s+1-j)-h_i(\tau s-j)\Vert\Big)^2\nonumber\\
\leq & \frac{\tau\sum^{K}_{k=1}\sum^{\tau}_{j=1}\Vert h_i(k\tau+1-j)-h_i(k\tau-j) \Vert^2}{(1-\sqrt{{\rho}})^2}\label{norm_analyize_form6}
\end{align}
according to Lemma \ref{ming_yan_lemma} and the relation $(\sum^{\tau}_{j=1}a_j)^2\leq\tau\sum^{\tau}_{j=1} a^2_j$.

Substituting  \eqref{norm_analyize_form5},  \eqref{norm_analyize_form7}, and \eqref{norm_analyize_form6} into \eqref{norm_analyize_form4} yields
\begin{align}\label{appendix_lemma1_add}
\sum^{T}_{t=1}\Vert y_i(t)\Vert^2\leq B_1+{\alpha}^2B_2\sum^{T-1}_{t=1}\Vert h_i(t)-h_i(t-1)\Vert^2,
\end{align}
where 
\begin{equation}\label{definition_B1B2}
\left\{
\begin{aligned}
B_1=&\frac{\tau(\tau+1)(2\tau+1) }{2(1-{\rho})}A^2_2 A^2_4,\\
B_2=&3\tau^4+\frac{\tau^2(\tau+1)(2\tau+1)(\rho\tau+\tau-1)^2}{2(1-\sqrt{{\rho}})^2} A^2_4,
\end{aligned}
\right.
\end{equation}
which establishes \eqref{lemma2_equation1}.

For the term  $\Vert h_i(t)-h_i(t-1)\Vert^2$ in \eqref{appendix_lemma1_add}, we have
\begin{align}
&\sum^{N}_{i=2}\Vert h_i(t-1)- h_i(t)\Vert^2\nonumber\\
\leq &\Vert \overline{\nabla f}(X(t))- \overline{\nabla f}(X(t-1))\Vert^2_{F}\nonumber\\
\leq& L^2\sum^{N}_{i=1} \Vert x_i(t)- x_i(t-1)\Vert^2,\label{add_1_lemma_complex1}
\end{align}
where the first and second inequalities follow from \eqref{definition_y_h} and Assumption \ref{smooth_assumption}, respectively. For the term $\sum^{N}_{i=1} \Vert x_i(t)- x_i(t-1)\Vert^2$ in \eqref{add_1_lemma_complex1}, we have
\begin{align}
& \sum^{N}_{i=1} \Vert x_i(t)- x_i(t-1)\Vert^2\nonumber\\
= & \sum^{N}_{i=1} \Vert Y(t)P^{\bf T}e(i)-Y(t-1)P^{\bf T}e(i)\Vert^2\nonumber\\
= & \sum^{N}_{i=1} \Vert y_i(t)- y_i(t-1)\Vert^2,\label{add_1_lemma_complex2}
\end{align}
where the first and second inequalities follow from \eqref{definition_y_h} and the orthogonal property $PP^{\top}=\mathbf{I}_N$, respectively. Substituting \eqref{add_1_lemma_complex2} into \eqref{add_1_lemma_complex1} completes the proof of Lemma \ref{computation_of_yi}.

\section{Proof of Theorem \ref{theorem1}}\label{proof_theorem1}
Under stepsize $0<{\alpha}\leq\frac{1}{\sqrt{5B_2}L}$, we have $1-{L{\alpha}}-\frac{{\alpha}^4 L^4 B_2}{B_3}>0$. Thus, from \eqref{convergence_property_last_step}, we have
\begin{align}
\frac{1}{T}\sum^{T}_{t=1}\Vert \nabla f(\overline{X}(t)) \Vert^2\leq& \frac{2\bigl\{ f(\overline{X}(1))-f(x^*)\bigr\}}{\tau{\alpha} K} + \frac{C_1 L^2}{B_3 N\tau K}.\label{appendixA_1}
\end{align}
For the second term on the right-hand side of \eqref{appendixA_1}, we have
\begin{align}
\frac{C_1L^2}{B_3N\tau}\leq&\frac{{\alpha}^2 L^4  B_2 ({\alpha}^2\Vert\overline{\nabla f}(X(0))\Vert^2+2A^2_2)}{B_3\tau}+\frac{B_1L^2}{B_3\tau}\label{analyze_term_appendixA1}
\end{align}
according to \eqref{definition_C1}.

The condition of stepsize $0<{\alpha}\leq\frac{1}{\sqrt{5B_2}L}$ implies $B_3=1-4{\alpha}^2L^2B_2\geq{\alpha}^2L^2B_2$, which further implies
\begin{align}
\frac{B_2}{B_3}\leq \frac{1}{{\alpha}^2 L^2}.\label{analyze_term_appendixA2}
\end{align}
Moreover, we have
\begin{align}
\frac{B_1}{B_3}\leq&\frac{1}{\alpha^2L^2}\frac{2(1-\sqrt{{\rho}})^2\tau(\tau+1)(2\tau+1) A^2_4 A^2_2}{2(1-{\rho})\rho^2\tau^4(\tau+1)(2\tau+1)A^2_4}\nonumber\\
\leq&\frac{ 2\Vert X(0)\Vert^2_{F}+2\Vert X(1)\Vert^2_{F}}{\alpha^2L^2\rho^2\tau^3},\label{analyze_term_appendixA3}
\end{align}
where the first inequality follows from \eqref{definition_B1B2} and \eqref{analyze_term_appendixA2}, and the second inequality follows from \eqref{many_definition_A} and $A^2_2\leq2\Vert X(0)\Vert^2_{F}+2\Vert X(1)\Vert^2_{F}$.

Substituting \eqref{analyze_term_appendixA1} and \eqref{analyze_term_appendixA3} into \eqref{appendixA_1} yields
\begin{align*}
&\frac{1}{T}\sum^{T}_{t=1}\Vert \nabla f(\overline{X}(t)) \Vert^2\nonumber\\
\leq& \frac{2}{\tau{\alpha} K} \big\{ f(\overline{X}(1))-f(x^*)\big\}+ \frac{ \Vert\overline{\nabla f}(X(0))\Vert^2}{\tau^3 K}\\
&+ \frac{4\Vert X(0)\Vert^2_{F}+4\Vert X(1)\Vert^2_{F}}{\tau^3\alpha^2 K}+\frac{ 2\Vert X(0)\Vert^2_{F}+2\Vert X(1)\Vert^2_{F}}{\alpha^2\rho^2\tau^4 K},
\end{align*}
which completes the proof of Theorem \ref{theorem1}.

\section{Proof of Theorem \ref{corollary1}}\label{proof_corollary1}
From \eqref{consensus_property} and \eqref{analyze_term_appendixA2}, we have
\begin{align}
&\frac{1}{NT}\sum^{N}_{i=1}\sum^{T}_{t=1}\Vert \overline{X}(t)-x_i(t)\Vert^2\nonumber\\
\leq&  \frac{C_1}{B_3NT}+\frac{{\alpha}^2}{T} \sum^{T-1}_{t=1}\Vert\overline{\nabla f}(X(t))\Vert^2,\label{corollary1_proof}
\end{align}
under stepsize $0<{\alpha}\leq\frac{1}{\sqrt{5B_2}L}$. In addition, from \eqref{analyze_term_appendixA2}, we have $1-{L{\alpha}}-\frac{{\alpha}^4 L^4 B_2}{B_3}>\frac{1}{2}.$ Substituting the above inequality into \eqref{convergence_property_last_step} yields
\begin{align}
 \sum^{T}_{t=1}\Vert \overline{\nabla f}(X(t))\Vert^2 \leq \frac{4}{{\alpha}} \bigr\{ f(\overline{X}(1))-f(x^*)\bigl\}+ \frac{2C_1 L^2}{B_3 N}.\label{corollary2_proof}
\end{align}
Substituting \eqref{corollary1_proof} into \eqref{corollary2_proof} yields
\begin{align}
&\frac{1}{NT}\sum^{N}_{i=1}\sum^{T}_{t=1}\Vert \overline{X}(t)-x_i(t)\Vert^2\nonumber\\
\leq&  \frac{3C_1}{B_3NT}+\frac{4{\alpha}}{T} \bigl\{ f(\overline{X}(1))-f(x^*)\bigr\}.\label{corollary3_proof}
\end{align}
Substituting \eqref{analyze_term_appendixA1} and \eqref{analyze_term_appendixA3} into \eqref{corollary3_proof} completes the proof of Theorem \ref{corollary1}.

\section{}\label{proof_theorem2}
\subsection{Proof of Theorem \ref{theorem2}}
For convenience of representation, we define 
\begin{equation}\label{definition_GT}
\left\{
    \begin{aligned}
    {G}(t)=&[\nabla f_1(x_1(t),\xi_1(t)),\cdots, \nabla f_N(x_N(t),\xi_N(t))],\\
    \overline{G}(t)=&\frac{1}{N}\sum^{N}_{i=1}\nabla f_i({x}_i(t),\xi_i(t)).
\end{aligned}
\right.
\end{equation}

Analogous to the processes \eqref{convergence_analysis_process1}-\eqref{average_xt_form}, we have 
\begin{align}\label{x_bar_appendix_Add1}
\overline{X}(t+1)=\overline{X}(t)-{\alpha} \overline{G}(t).
\end{align}
Combining Assumption \ref{smooth_assumption} and \eqref{x_bar_appendix_Add1} leads to
\begin{align*}
&f(\overline{X}(t+1))\\
\leq & f(\overline{X}(t))-\langle \nabla f(\overline{X}(t)),{\alpha} \overline{G}(t)\rangle+\frac{L{\alpha^2}}{2}\Vert \overline{G}(t)\Vert^2.
\end{align*}
Taking the expectation on both sides of the above inequality yields
\begin{align}
&\mathbb{E}[f(\overline{X}(t+1))]\nonumber\\
\leq & \mathbb{E}[f(\overline{X}(t))]-{\alpha} \mathbb{E}[\langle \nabla f(\overline{X}(t)),\overline{\nabla f}(X(t))\rangle]\nonumber\\
&+\frac{L{\alpha^2}}{2}\mathbb{E}[\Vert \overline{\nabla f}(X(t))\Vert^2]+\frac{L{\alpha^2}}{2}\mathbb{E}[\Vert \overline{G}(t)-\overline{\nabla f}(X(t))\Vert^2]\nonumber\\
\leq & \mathbb{E}[f(\overline{X}(t))]-\frac{{\alpha}}{2} \mathbb{E}[\Vert \nabla f(\overline{X}(t))\Vert^2]\nonumber\\
&-\bigl(\frac{{\alpha}}{2} -\frac{L{\alpha^2}}{2}\bigr)\mathbb{E}[\Vert\overline{\nabla f}(X(t))\Vert^2]\nonumber\\
&+\frac{{\alpha}}{2} \mathbb{E}[\Vert\nabla f(\overline{X}(t))-\overline{\nabla f}(X(t))\Vert^2]+\frac{L{\alpha^2}}{2}\sigma^2,\label{add_stochastic_averaging_1}
\end{align}
where the first inequality follows from \eqref{noise1_assumption} in Assumption \ref{stochastic_gradient}, and the second inequality follows from \eqref{noise2_assumption} in Assumption \ref{stochastic_gradient} and the relation $2\langle a,b\rangle=\Vert a+b\Vert^2-\Vert a\Vert^2-\Vert b\Vert^2$. 

For the term $\mathbb{E}[\Vert\nabla f(\overline{X}(t))-\overline{\nabla f}(X(t))\Vert^2]$ in \eqref{add_stochastic_averaging_1}, we have
\begin{align}
&\mathbb{E}[\Vert\nabla f(\overline{X}(t))-\overline{\nabla f}(X(t))\Vert^2]\nonumber\\
\leq &\frac{1}{N}\sum^{N}_{i=1}\mathbb{E}[\Vert \nabla f_i(\overline{X}(t))-\nabla f_i({x}_i(t))\Vert^2]\nonumber\\
\leq &\frac{L^2}{N}\sum^{N}_{i=1}\mathbb{E}[\Vert \overline{X}(t)-{x}_i(t)\Vert^2],\label{add_stochastic_averaging_2}
\end{align}
where the first and second inequalities follow from the relation $\Vert \sum^{N}_{i=1}a_i\Vert^2\leq N\sum^{N}_{i=1}\Vert a_i\Vert^2$ and  Assumption \ref{smooth_assumption}, respectively. Substituting \eqref{add_stochastic_averaging_2} into \eqref{add_stochastic_averaging_1} yields
\begin{align}
&\frac{{\alpha}}{2} \mathbb{E}[\Vert \nabla f(\overline{X}(t))\Vert^2]+\bigl(\frac{{\alpha}}{2} -\frac{L{\alpha^2}}{2}\bigr)\mathbb{E}[\Vert\overline{\nabla f}(X(t))\Vert^2]\nonumber\\
\leq & \mathbb{E}[f(\overline{X}(t))]-\mathbb{E}[f(\overline{X}(t+1))]+\frac{L{\alpha^2}}{2}\sigma^2\nonumber\\
&+ \frac{{\alpha} L^2}{2N}\sum^{N}_{i=1}\mathbb{E}[\Vert \overline{X}(t)-{x}_i(t)\Vert^2],\label{important_most_stochastic_average_property}
\end{align}
which further implies
\begin{align}
& \sum^{T}_{t=1}\bigr\{\mathbb{E}[\Vert \nabla f(\overline{X}(t))\Vert^2]+(1 -{L{\alpha}})\mathbb{E}[\Vert\overline{\nabla f}(X(t))\Vert^2]\bigr\}\nonumber\\
\leq &\frac{2}{{\alpha}}\bigl\{ \mathbb{E}[f(\overline{X}(1))]-f(x^*)\bigr\}+{L{\alpha}}\sigma^2 T\nonumber\\
&+\frac{ L^2}{N}\sum^{T}_{t=1}\sum^{N}_{i=1}\mathbb{E}[\Vert \overline{X}(t)-{x}_i(t)\Vert^2]\nonumber\\
\leq & \frac{2}{{\alpha}}\bigl\{ \mathbb{E}[f(\overline{X}(1))]-f(x^*)\bigr\}+ {L{\alpha}}\sigma^2 T+\frac{12B_2L^2}{B_4}{\alpha^2}\sigma^2T\nonumber\\
&+\frac{C_2 L^2}{B_4N}+\frac{6{\alpha}^4L^4B_2}{B_4}\sum^{T-1}_{t=1}\mathbb{E}[\Vert\overline{\nabla f}(X(t))\Vert^2],\label{add_more_stochastic1}
\end{align}
where the first and second inequalities follow from (\ref{important_most_stochastic_average_property}) and Lemma \ref{consensus_stochastic}, respectively.

If a stepsize satisfies $0<\alpha\leq\frac{1}{\sqrt{13B_2}L}$, we have 
\begin{equation}\label{last_term}
\left\{
\begin{aligned}
&B_4\geq {\alpha^2}L^2 B_2,\\
&1 -{L{\alpha}}-\frac{6{\alpha}^4L^4B_2}{B_4}>\frac{1}{2},\\
&1-12\alpha^2B_2L^2\geq \frac{1}{13}.
\end{aligned}
\right.
\end{equation}
Substituting \eqref{last_term} into \eqref{add_more_stochastic1} yields
\begin{align}
 \frac{1}{T}\sum^{T}_{t=1}\mathbb{E}[\Vert& \nabla f(\overline{X}(t))\Vert^2]\leq  \frac{2}{{\alpha} T}\bigl\{ \mathbb{E}[f(\overline{X}(1))]-f(x^*)\bigr\} \nonumber\\
&+\frac{C_2 L^2}{B_4NT}+ {L{\alpha}}\sigma^2+{156B_2}{\alpha^2}L^2\sigma^2.\label{appendixB1}
\end{align}
For the term $\frac{C_2 L^2}{B_4NT}$ in \eqref{appendixB1}, we have
\begin{align}
&\frac{C_2 L^2}{B_4N\tau K}\nonumber\\
\leq&\frac{6{\alpha}^4 L^4  B_2 \mathbb{E}[\Vert\overline{\nabla f}(X(0))\Vert^2]}{B_4\tau K}+\frac{6{\alpha}^2 L^4  B_2 A^2_5}{B_4\tau K}+\frac{B_1L^2}{B_4\tau K}\nonumber\\
\leq&\frac{6  \mathbb{E}[\Vert\overline{\nabla f}(X(0))\Vert^2]}{\tau^3 K}+\frac{12\mathbb{E}[\Vert X(1)\Vert_F^2]+12\mathbb{E}[\Vert X(0)\Vert_F^2]}{\tau^3\alpha^2 K}\nonumber\\
&+\frac{B_1L^2}{B_4\tau K},\label{appendixB2}
\end{align}
where the first and second inequalities follow from \eqref{definition_c2a5} and \eqref{last_term}, respectively. Moreover, from \eqref{analyze_term_appendixA3} and \eqref{last_term}, we have
\begin{align}
\frac{B_1}{B_4}\leq\frac{ 2\mathbb{E}[\Vert X(0)\Vert^2_{F}]+2\mathbb{E}[\Vert X(0)\Vert^2]}{\alpha^2L^2\rho^2\tau^3}.\label{appendixB3}
\end{align}
Substituting \eqref{appendixB2} and \eqref{appendixB3} into \eqref{appendixB1} completes the proof of Theorem \ref{theorem2}.

\subsection{Supporting Lemmas for the Proof of Theorem \ref{theorem2}}
\begin{lemma}\label{consensus_stochastic}
Under stepsize $0<\alpha\leq\frac{1}{\sqrt{13B_2}L}$, we have
\begin{align}
&\frac{1}{NT}\sum^{N}_{i=1}\sum^{T}_{t=1}\mathbb{E}[\Vert \overline{X}(t)-x_i(t)\Vert^2]\leq  \frac{C_2}{B_4NT}\nonumber\\
&+\frac{6{\alpha}^4L^2B_2}{B_4T}\sum^{T-1}_{t=1}\mathbb{E}[\Vert\overline{\nabla f}(X(t))\Vert^2]+\frac{12B_2}{B_4}{\alpha^2}\sigma^2,\label{last_lemma_stochastic}
\end{align}
where
\begin{equation}\label{definition_c2a5}
\left\{
\begin{aligned}
C_2=&6{\alpha}^2 L^2 N B_2 ({\alpha}^2\mathbb{E}[\Vert\overline{\nabla f}(X(0))\Vert^2]+A^2_5)+NB_1,\\
A_5=&2\mathbb{E}[\Vert X(1)\Vert^2]+2\mathbb{E}[\Vert X(0)\Vert^2],\\
B_4=&1-12{\alpha^2}L^2 B_2.
\end{aligned}
\right.
\end{equation}
\end{lemma}
\begin{proof}
Algorithm \ref{algorithm_recur} can be equivalently expressed as
\begin{align}
X(t{+}1) &= 2X(t)W(t) - X(t{-}1)W(t) \nonumber\\
&\quad - \alpha G(t)W(t) + \alpha  G(t{-}1)W(t), \label{convergence_analysis_process1_stochastic}
\end{align}
where $X(t)$, $W(t)$, and $G(t)$ are defined in \eqref{matrix_definition}, \eqref{matrix_W}, and \eqref{definition_GT}, respectively. Furthermore, we define
\begin{equation}\label{definition_y_h_stochastic}
\left\{
\begin{aligned}
Y(t) &= X(t)P = [y_1(t), y_2(t), \cdots, y_N(t)],\\
H(t) &= G(t)P = [h_1(t), h_2(t), \cdots, h_N(t)],
\end{aligned}
\right.
\end{equation}
where $P$ is the  orthogonal matrix satisfying $P^{\top}P = PP^{\top} = \mathbf{I}_N$ and $W(t)= P \Lambda(t) P^{\top}$ (see \eqref{matrix_form_time_varying}). Following a similar approach to that used in deriving \eqref{appendix_lemma1_add}, we obtain 
\begin{align}
\sum^{T}_{t=1}\mathbb{E}[\Vert y_i(t)\Vert^2]\leq& B_1+{\alpha}^2B_2\sum^{T-1}_{t=1}\mathbb{E}[\Vert h_i(t)-h_i(t-1)\Vert^2],\label{very_important_inequality1_stochastic}
\end{align}
where $B_1$ and $B_2$ are defined in \eqref{definition_B1B2}. For the term $\sum^{N}_{i=2} \mathbb{E}[\Vert h_i(t)-h_i(t-1)\Vert^2]$ in \eqref{very_important_inequality1_stochastic}, we can obtain
\begin{align*}
\sum^{N}_{i=2} \mathbb{E}[\Vert h_i(t)-h_i(t-1)\Vert^2]\leq \mathbb{E}[\Vert G(t)-G(t-1)\Vert^2_F]
\end{align*}
from \eqref{definition_y_h_stochastic}, the definition of $\Vert\cdot\Vert_F$, and the orthogonal property $P P^\top=\mathbf{I}_N$. Applying \eqref{definition_GT}, \eqref{definition_y_h_stochastic}, and the relation $\Vert a+b+c\Vert^2\leq 3\Vert a\Vert^2+3\Vert b\Vert^2+3\Vert c\Vert^2$ to the above inequality yields
\begin{align}
&\sum^{N}_{i=2} \mathbb{E}[\Vert h_i(t)-h_i(t-1)\Vert^2]\nonumber\\
\leq&3\sum^{N}_{i=1} \mathbb{E}[\Vert \nabla f_i(x_i(t-1),\xi_i(t-1))- \nabla f_i(x_i(t-1))\Vert^2]\nonumber\\
&+3\sum^{N}_{i=1} \mathbb{E}[\Vert \nabla f_i(x_i(t),\xi_i(t))- \nabla f_i(x_i(t))\Vert^2]\nonumber\\
&+3\sum^{N}_{i=1} \mathbb{E}[\Vert \nabla f_i(x_i(t))- \nabla f_i(x_i(t-1))\Vert^2].\label{very_important_inequality1_stochastic2_add2}
\end{align}
Furthermore, we can obtain
\begin{align}
&\sum^{N}_{i=2} \mathbb{E}[\Vert h_i(t)-h_i(t-1)\Vert^2]\nonumber\\
\leq& 6\sigma^2 N+3L^2 \sum^{N}_{i=1}\mathbb{E}[\Vert y_i(t)- y_i(t-1)\Vert^2]\label{very_important_inequality1_stochastic2}
\end{align}
based on Assumption \ref{stochastic_gradient}, \eqref{very_important_inequality1_stochastic2_add2}, \eqref{definition_y_h_stochastic}, and the orthogonal property $P P^\top=\mathbf{I}_N$. Substituting (\ref{very_important_inequality1_stochastic2}) into (\ref{very_important_inequality1_stochastic}) leads to
\begin{align}
&\sum^{N}_{i=2}\sum^{T}_{t=1}\mathbb{E}[\Vert y_i(t)\Vert^2]\leq6{\alpha^2} NB_2\sigma^2T+NB_1\nonumber\\
& \quad+3{\alpha^2}L^2 B_2\sum^{T-1}_{t=1}\sum^{N}_{i=1}\mathbb{E}[\Vert y_i(t)- y_i(t-1)\Vert^2].\label{appendix_d_1}
\end{align}

Then, we need to analyze the term $\mathbb{E}[\Vert y_1(t)- y_1(t-1)\Vert^2]$ in \eqref{appendix_d_1}. From \eqref{orthoghnal_matrix}, \eqref{lambda1_rho1_property}, and \eqref{definition_y_h_stochastic}, we can obtain
\begin{align}\label{y_1_t_overline_stochastic}
y_1(t)=X(t)v_1=\sqrt{N} \overline{X}(t).
\end{align}
Combining \eqref{x_bar_appendix_Add1} and the relation $\Vert a+b\Vert^2\leq2\Vert a\Vert^2+2\Vert b\Vert^2$ leads to
\begin{align*}
&\Vert \overline{X}(t+1)-\overline{X}(t)\Vert^2\nonumber\\
\leq& 2{\alpha^2} \Vert\overline{G}(t)-\overline{\nabla f}(X(t))\Vert^2+2{\alpha^2} \Vert\overline{\nabla f}(X(t))\Vert^2,
\end{align*}
which further yields 
\begin{align*}
\mathbb{E}[\Vert \overline{X}(t+1)-\overline{X}(t)\Vert^2]\leq  2{\alpha^2}\bigl(\sigma^2+ \mathbb{E}[\Vert\overline{\nabla f}(X(t))\Vert^2]\bigr)
\end{align*}
based on Assumption \ref{stochastic_gradient}. Thus, from \eqref{y_1_t_overline_stochastic} and  the above inequality, we have
\begin{align*}
\mathbb{E}[\Vert y_1(t+1)- y_1(t) \Vert^2]\leq 2{\alpha^2} N\bigl(\sigma^2+\mathbb{E}[\Vert\overline{\nabla f}(X(t))\Vert^2]\bigr).
\end{align*}
Substituting the above inequality into \eqref{appendix_d_1} and applying the inequality $\Vert a+b\Vert^2\leq2\Vert a\Vert^2+2\Vert b\Vert^2$ yield
\begin{align*}
&\sum^{N}_{i=2}\sum^{T}_{t=1}\mathbb{E}[\Vert y_i(t)\Vert^2]\\
\leq & NB_1+6{\alpha^2}L^2 B_2\sum^{T-1}_{t=1}\sum^{N}_{i=2}\mathbb{E}[\Vert y_i(t)\Vert^2+ \Vert y_i(t-1)\Vert^2]\\
&+6{\alpha}^4L^2B_2N\sum^{T-1}_{t=1}\mathbb{E}[\Vert\overline{\nabla f}(X(t-1))\Vert^2]+12B_2N{\alpha^2}\sigma^2T.
\end{align*}
Substituting \eqref{consensus_analysis_important} into the above inequality and rearranging terms yield
\begin{align*}
&B_4\sum^{N}_{i=1}\sum^{T}_{t=1}\mathbb{E}[\Vert \overline{X}(t)-x_i(t)\Vert^2]\\
\leq & C_2+6{\alpha}^4L^2B_2N\sum^{T-1}_{t=1}\mathbb{E}[\Vert\overline{\nabla f}(X(t))\Vert^2]+12B_2N{\alpha^2}\sigma^2T,
\end{align*}
where $B_4$ and $C_2$ are defined in \eqref{definition_c2a5}. Moreover, under stepsize $0<\alpha\leq\frac{1}{\sqrt{13B_2}L}$, we have $B_4>0$. To complete the proof of Lemma \ref{consensus_stochastic}, we divide both sides of the above inequality by \( B_4 \). 
\end{proof}

\section{Proof of Theorem \ref{corollary2}}\label{proof_corollary2}
Rearranging terms of \eqref{add_more_stochastic1} and applying the relation \eqref{last_term} yield
\begin{align}
 \frac{1}{T}\sum^{T}_{t=1}&\mathbb{E}[\Vert\overline{\nabla f}({X}(t))\Vert^2]\leq  \frac{4}{{\alpha} T}\bigl\{ \mathbb{E}[f(\overline{X}(1))]-f(x^*)\bigr\} \nonumber\\
&+\frac{2C_2 L^2}{B_4NT}+ {L{\alpha}}\sigma^2+\frac{24B_2}{B_4}{\alpha^2}L^2\sigma^2.\label{corollary2_1}
\end{align}
Substituting \eqref{corollary2_1} into \eqref{last_lemma_stochastic} yields
\begin{align}
&\frac{1}{NT}\sum^{N}_{i=1}\sum^{T}_{t=1}\mathbb{E}[\Vert \overline{X}(t)-x_i(t)\Vert^2]\leq  \frac{156B_2}{B_4}{\alpha^2}\sigma^2\nonumber\\
 &+ {6{\alpha^2}}\sigma^2+\frac{13C_2}{B_4NT} +\frac{24{\alpha}}{ T}\bigl\{ \mathbb{E}[f(\overline{X}(1))]-f(x^*)\bigr\}.\label{corollary2_3}
\end{align}
Substituting \eqref{last_term}, \eqref{appendixB2}, and \eqref{appendixB3} into \eqref{corollary2_3} completes the proof of Theorem \ref{corollary2}.

\bibliographystyle{ieeetr}
\bibliography{reference}

\end{document}